\documentclass[a4paper,12pt]{amsart}
\usepackage{amsthm}
\usepackage{amssymb}
\usepackage{amsmath}
\usepackage{verbatim,enumerate,url}
\renewcommand{\thefootnote}{\fnsymbol{footnote}}

\newtheorem{theorem}{Theorem}[section]
\newtheorem{lemma}[theorem]{Lemma}

\newtheorem{proposition}[theorem]{Proposition}
\newtheorem{remark}[theorem]{Remark}

\newtheorem{example}[theorem]{Example}

\newtheorem*{example*}{Example}

\newtheorem*{theorem*}{Theorem}

\newtheorem*{remark*}{Remark}

\newtheorem{corollary}[theorem]{Corollary}
\newtheorem*{corollary*}{Corollary}

\newtheorem{definition}[theorem]{Definition}

\newtheorem*{definition*}{Definition}

\newtheorem*{notation*}{Notation}

\newtheorem{notation}[theorem]{Notation}

\numberwithin{equation}{section}

\gdef\myletter{}

\let\savetheequation\theequation
\def\theequation{\savetheequation\myletter}

\newcommand{\CC}{{\mathbb C}}

\newcommand{\RR}{{\mathbb R}}

\newcommand{\PP}{{\mathbb P}}

\newcommand{\NN}{{\mathbb N}}

\newcommand{\Res}{\mathrm{Res}}

\newcommand{\calA}{\mathcal{A}}
\newcommand{\calM}{\mathcal{M}}
\newcommand{\calB}{\mathcal{B}}
\newcommand{\calF}{\mathcal{F}}
\newcommand{\calG}{\mathcal{G}}

\newcommand{\calI}{\mathcal{I}}
\newcommand{\calC}{\mathcal{C}}

\newcommand{\bA}{\mathbf{A}}
\newcommand{\bR}{\mathbf{R}}
\newcommand{\bW}{\mathbf{W}}
\newcommand{\bZ}{\mathbf{Z}}

\newcommand{\Van}{{\mathrm{Van}}}

\renewcommand{\O}{{\mathcal O}}

\newcommand{\be}{{\bf e}}

\def \hat{\widehat}

\def \bw{{\bf w}}
\def \b0{{\bf 0}}
\def \lot{\hbox{\it l.o.t.}}

\def \span{\mathrm{span}}

\long\def\symbolfootnote[#1]#2{\begingroup%
\def\thefootnote{\fnsymbol{footnote}}\footnote[#1]{#2}\endgroup}

\begin{document}

\title[Transfinite diameter on the graph of a polynomial mapping]{Transfinite diameter on the graph of a polynomial mapping and the DeMarco-Rumely formula}
\author{Sione Ma`u}

\address{Department of Mathematics,
University of Auckland,
Auckland, NZ}
\email{s.mau@auckland.ac.nz}

\begin{abstract}
We study Chebyshev constants and transfinite diameter on the graph of a polynomial mapping $f\colon\CC^2\to\CC^2$.   We show that two transfinite diameters of a compact subset of the graph (i.e., defined with respect to two different collections of monomials) are equal when the set has a certain symmetry.  As a consequence, we give a new proof in $\CC^2$ of a pullback formula for transfinite diameter due to DeMarco and Rumely that involves a homogeneous resultant.

\end{abstract}

\maketitle

\section{Introduction and summary}

In this paper we study transfinite diameter on a graph, i.e., an affine algebraic variety 
\begin{equation}\label{eqn:Vf}V=\{(w,z)\in\CC^2\times\CC^2\colon w=f(z)\}\end{equation} where  $f\colon\CC^2\to\CC^2$ is a polynomial mapping.   The main goal of the paper is to give a direct proof  of the pullback formula of DeMarco and Rumely for transfinite diameter in $\CC^2$.   Their main result in \cite{demarcorumely:transfinite} is the following:

\begin{theorem*}
Let $f:\CC^N\to\CC^N$ be a regular polynomial map of degree $d$ (where $d\in\NN$).\footnote{See the beginning of Section \ref{sec:pullback} for the definition of a regular polynomial map in $\CC^2$.}  For any compact set $K\subset\CC^N$, we have the formula
$$
d(f^{-1}(K)) = |\Res(\hat f)|^{-1/(Nd^N)}d(K)^{1/d},
$$ 
where $\hat f$ denotes the degree $d$ homogeneous part of $f$.
\end{theorem*}

The original proof utilized deep results in complex dynamics and arithmetic intersection theory; specifically, a formula of Bassanelli and Berteloot relating the resultant of a regular homogeneous polynomial mapping on $\CC^N$ to the Green current of the induced mapping on $\PP^{N-1}$ \cite{bassanelliberteloot:bifurcation}; a formula of Rumely relating the transfinite diameter of $K$ to its Robin function; and a pullback formula for a global sectional capacity developed by Rumely and his collaborators \cite{rumelylauvarley:existence}.  Using these ingredients, the theorem was first proved in the special case where $K$ is the Julia set of a polynomial mapping, then extended to general compact sets by some clever approximation arguments.

In Theorem \ref{thm:resthm} we prove the case $N=2$ using only basic notions of pluripotential theory and Zaharjuta's techniques for estimating Vandermonde determinants,  combined with computational algebraic geometry and linear algebra.  The resultant term in the pullback formula appears naturally as a result of eliminating variables, and connects the DeMarco-Rumely formula (at least the $\CC^2$ version of it) more directly to the ideas of elimination theory.  A general proof (in $\CC^N$) using the methods of this paper seems too unwieldly as the resultant becomes more and more complicated in higher dimensions.

We now give a summary of the paper.  Let $V$ be the graph of a polynomial mapping $f$ as in (\ref{eqn:Vf}).  A polynomial in the coordinate ring $\CC[V]$ may be identified with a polynomial in $\CC[z]$ by elimination.  It may also be identified with a normal form in $\CC[w,z]_V$, defined in Section \ref{sec:nform}.

In Section \ref{sec:cheby} we study directional Chebyshev constants on $V$ defined by a limiting process using normal forms.  The methods are similar to \cite{coxmau:transfinite}, but simpler.  A technical  algebraic  condition (property $(\star)$) is useful to establish convergence.

In Section \ref{sec:invariance} we study the symmetry of sets of the form 
\begin{equation}\label{eqn:Lintro}L=\{(w,z)\in V \colon w\in K\}\end{equation} where $K\subset\CC^2$ is compact. Our main theorem of this section (Theorem \ref{thm:invariance}) uses this symmetry to show that two different types of directional Chebyshev constants are equivalent.

In Section \ref{sec:transfinite} we study transfinite diameters $d^{(1)}(K)$, $d^{(2)}(K)$ and $d^{(3)}(K)$ of a set $K\subset V$ using various collections of monomials. 
  The main theorem (Theorem \ref{thm:zaharj}) gives integral formulas for each of $d^{(j)}(K)$ in terms of directional Chebyshev constants.  Actually, a proof is only sketched for $d^{(2)}(K)$, since $d^{(1)}(K),d^{(3)}(K)$ may be identified with the classical (Fekete-Leja) transfinite diameters of the projections to the $z$ and $w$ axes respectively, whose formulas were established by Zaharjuta \cite{zaharjuta:trans}.   When $L$ is as in (\ref{eqn:Lintro}), these formulas together with Theorem \ref{thm:invariance} yield $d^{(2)}(L)=d^{(3)}(L)$ as a consequence.

In Section \ref{sec:pullback} we continue with $V$ being the graph of a regular polynomial mapping and $L$ as in Section \ref{sec:invariance}.  We use $w=f(z)$ to substitute for the $w$ in terms of $z$ in a basis $\calC$ of monomials on $V$, thereby   relating $d^{(1)}(L)$ to $d^{(2)}(L)$.  Powers of the resultant $\Res(\hat f)$ (where $\hat f$ is the  leading homogeneous part of $f$) arise in the process. When $L$ is as in (\ref{eqn:Lintro}) this gives a relation between $d^{(1)}(L)$ and $d^{(3)}(L)$.  Identifying these last two quantities with the Fekete-Leja transfinite diameters of the projections to $z$ and $w$ respectively, we recover the DeMarco-Rumely formula in $\CC^2$.  

\section{Normal Form}\label{sec:nform}
We consider an affine variety given by a graph over $\CC^2$ of polynomials of degree $d_1\geq d_2\geq 1$.    Let 
\begin{equation}\label{eqn:f}
f_1(z):= \sum_{\substack{j,k\geq 0\\ j+k\leq d_1}} a_{jk}z_1^iz_2^j,\quad f_2(z):=\sum_{\substack{j,k\geq 0\\ j+k\leq d_2}} b_{jk}z_1^jz_2^k,
\end{equation}
be irreducible polynomials 
and let 
\begin{equation}\label{eqn:V}
V=\{(w,z)\in\CC^2\times\CC^2\colon w_1=f_1(z),\  w_2=f_2(z)\}.
\end{equation}
Also, define the polynomial map $f=(f_1,f_2)$.  Then  $V$ is the graph of $w=f(z)$.  Let $\CC[V]$ denote the coordinate ring of $V$, i.e., the polynomials restricted to $V$. 

\begin{lemma}
The monomials $w^{\alpha}=w_1^{\alpha_1}w_2^{\alpha_2}$ are linearly independent in $\CC[V]$.
\end{lemma}

\begin{proof}
Consider  a linear combination
$$
R(w):=\sum_{\alpha} c_{\alpha}w^{\alpha}=0.
$$
Let $U\subset V$ be a connected open set on which the projection $V\ni(w,z)\mapsto w\in \CC^2$ is holomorphic.  Using this, identify $R(w)$ with a holomorphic function on $U_w\subset\CC^2$ which is the image of $U$ under the projection.  Since $R(w)$ is identically zero on $U_w$ and $U_w$ is open, we must have $c_{\alpha}=0$ for all $\alpha$.  
\end{proof}

We now construct a basis $\calB$ of monomials on $V$ using a grevlex ordering.  Order the  monomials in $\CC[w,z]=\CC[w_1,w_2,z_1,z_2]$ first by total degree: $w^{\alpha}z^{\beta}\prec w^{\tilde\alpha}z^{\tilde\beta}$ whenever $|\alpha|+|\beta|<|\tilde\alpha|+|\tilde\beta|$.  Then order monomials with the same total degree lexicographically, according to $w_1\prec w_2\prec z_1 \prec z_2$.

The basis $\calB$ is then constructed by going through the monomials listed in increasing grevlex order, and removing any monomial that is linearly dependent with respect to earlier monomials in the list.  The \emph{normal form} of a polynomial $p\in\CC[V]$ is its representation as a linear combination of basis elements. 

\begin{lemma}
Every monomial in $\CC[w]$  is an element of $\calB$.  Moreover, there is a finite collection of monomials $\calI\subset\CC[z]$ (which includes $1$) such that any monomial in $\calB$ is of the form $z^{\beta}w^{\alpha}$ for some $z^{\beta}\in\calI$.  Hence a normal form is an element of
$$
\bigoplus_{z^{\beta}\in\calI} z^{\beta}\CC[w]
$$
\end{lemma}

\begin{proof}
In the linear case ($d_1=d_2=1$) any monomial containing $z_1,z_2$ is removed, and a normal form is simply an element of $\CC[w]$.

When $d_1\geq d_2>1$, then within each fixed total degree, all monomials of the form $w^{\alpha}$ are not removed since they are linearly independent by the previous lemma and listed ahead of monomials containing $z$ for that same degree.  The monomials that are removed are monomials containing $z$ that are leading monomials of elements of the ideal $I=\langle f_1-w_1,f_2-w_2\rangle\subset\CC[w,z]$.  Observe that the two expressions
$$g_1(f_1-w_1) +g_2(f_2-w_2),\quad g_1f_1+g_2f_2\qquad (\hbox{where } g_1,g_2\in\CC[z])$$
have the same leading terms in our grevlex ordering.  In other words, leading monomials of elements of $\langle f_1,f_2\rangle$ as an ideal in $\CC[z]$ are also leading monomials of elements of $I$.   The number of monomials that are not leading monomials of elements of $\langle f_1,f_2\rangle$  is finite because the number of points in the common zero set of $f_1,f_2$ in $\CC^2$ is finite (see e.g. \cite{coxlittleoshea:ideals}, 5\S 3).\footnote{When $I$ is a radical ideal the two numbers are equal.}   This finite set of monomials, by definition, is $\calI$.  
\end{proof}

\begin{example}\label{ex:13}  \rm 
Suppose 
$$\begin{aligned}
f_1(z)&=a_dz_1^d+a_{d-1}z_1^{d-1}z_2+\cdots+a_0z_2^d+\lot,\\
f_2(z)&=b_dz_1^d+b_{d-1}z_1^{d-1}z_2+\cdots+b_0z_2^d+\lot
\end{aligned}$$
where $\lot$ denotes terms of total degree strictly less than $d$.  For a generic choice of the coefficients $a_j,b_j$,
$\calI$ is the set of monomials not in $\langle z_1z_2^{d-1}, z_2^d\rangle$.  
\end{example}

\begin{notation}\rm Normal forms constructed with respect to the basis $\calB$ will be denoted $\CC[w,z]_V$.  
\end{notation}

 \section{Chebyshev constants}\label{sec:cheby}

Let $V$ be defined as in (\ref{eqn:f}), (\ref{eqn:V}).  Let $\calI$  denote the finite collection of multi-indices $\beta$ such that $w^{\alpha}z^{\beta}\in\CC[w,z]_V$ for some multi-index $\alpha\in\NN_0^2$.  We will assume that the basis of monomials $\calB:=\{w^{\alpha}z^{\beta}\colon \alpha\in\NN_0^2,\ \beta\in\calI\}$ is ordered so that the following property holds:
\begin{itemize}
\item[($\star$)] For each $\beta\in\calI$ there exists an index $\tilde\beta$ and a constant $C_{\beta}\in\RR$ such that 
$
z^{\beta}\cdot z^{\tilde\beta} = C_{\beta}w^{\gamma} +\lot 
$ for some $\gamma\in\NN_0^2$.  
\end{itemize}
Here $\lot$ refers to terms of lower order according to our ordering. 

\begin{example}\label{ex:21}  \rm
Suppose $V$ is given by $w=f(z)$ where $f=(f_1,f_2)$ with components of degree $d$ as in Example \ref{ex:13}.  Consider the ordering on monomials in $\CC[w,z]_V$ for which $w^{\alpha}z^{\beta}\prec w^{\tilde\alpha}z^{\tilde\beta}$ if 
\begin{itemize}
\item $d|\alpha|+|\beta|<d|\tilde\alpha|+|\tilde\beta|$; or
\item $d|\alpha|+|\beta|=d|\tilde\alpha|+|\tilde\beta|$ and $\alpha\prec\tilde\alpha$; or
\item $d|\alpha|+|\beta|=d|\tilde\alpha|+|\tilde\beta|$, $\alpha=\tilde\alpha$ and $\beta\prec\tilde\beta$.  
\end{itemize}
In the generic case of Example \ref{ex:13},  $(\star)$ holds, and we can take $z^{\tilde\beta}$ to be a sufficiently large power of $z_2$.

We illustrate the case $d=2$; here  
$$w_1=a_2z_1^2+a_1z_1z_2+a_0z_2^2+\lot, \ w_2=b_2z_1^2+b_1z_1z_2+b_0z_2^2+\lot. $$
We have $b_0w_1-a_0w_2= (a_1b_0-a_0b_1)z_1z_2+\lot$.  Generically, $a_1b_0-a_0b_1\neq 0$, which means $z_1z_2$ and $z_2^2$ are not basis monomials, and we have $\calI=\{1,z_1,z_2,z_1^2\}$.  The first few monomials listed according to the ordering are
$$
1,z_1,z_2,z_1^2,w_1,w_2,z_1w_1,z_2w_1,z_1w_2,z_2w_2,z_1^2w_1,z_1^2w_2,w_1^2,w_1w_2,w_2^2,\ldots
$$
To verify $(\star)$, take, say, the monomial $z_1^2$; then multiplying by $z_2^2$, we obtain
$$
z_1^2\cdot z_2^2 = (z_1z_2)^2 = (b_0w_1 + a_0w_2)^2+ \lot = a_0^2w_2^2 + \lot.
$$
For the monomials $1,z_1,z_2$, multiplying by $z_2$ works.
\end{example}

   Let $K\subset V$ be a compact set.  
Define $Y_K\colon \NN_0^2\to[0,\infty)$ by 
 $$
Y_K(\alpha) := \inf\{\|p\|_K\colon p\in\CC[w,z]_V,\ p(w,z) = w^{\alpha} +  \lot  \}.
$$
It is easy to verify that 
$Y_K$ is \emph{submultiplicative}, i.e., 
$$
Y_K(\alpha+\beta)\leq Y_K(\alpha)Y_K(\beta) \hbox{ for all } \alpha,\beta\in\NN_0^2.
$$

Before stating the next result we first recall some notation.  Let $\Sigma\subset\RR^2$ denote the closed line segment joining $(1,0)$ to $(0,1)$, and let $\Sigma^{\circ}=\Sigma\setminus\{(1,0),(0,1)\}$.  Let $h(s)$ denote the number of monomials of degree $s$, i.e., the number of elements in $\{ \alpha\in\NN_0^2\colon  |\alpha|=s\}$.  
As a consequence of the limiting properties of submultiplicative functions (see \cite{bloomlev:weighted}), we have the following result.

\begin{proposition}\label{prop:11}
The limit 
$$
T_{\calB}(K,\theta):=\lim_{\substack{|\alpha|\to\infty\\ \alpha/|\alpha|\to\theta}} Y_K(\alpha)^{1/|\alpha|}
$$
exists for each $\theta\in \Sigma^{\circ}$, and $\theta\mapsto T_{\calB}(K,\theta)$ defines a logarithmically convex function on $\Sigma^{\circ}$.  Moreover, we have the convergence
$$
\frac{1}{h(s)}\sum_{|\alpha|=s}\log Y_K(\alpha)^{1/|\alpha|} \ \to \  \int_{0}^1\log T_{\calB}(K,\theta(t))dt \quad\hbox{as }s\to\infty,  
$$
where $\theta(t)=(t,1-t)$.  \qed
\end{proposition}

\begin{definition}\rm
We call the function $\theta\mapsto T_{\calB}(K,\theta)$ the \emph{Chebyshev transform} of $K$ with respect to the basis $\calB$.   A fixed $\theta$ is called a \emph{direction} and  $T_{\calB}(K,\theta)$ is a \emph{directional Chebyshev constant}.  
\end{definition}
The above terminology is based on \cite{nystrom:transforming} and \cite{bloomlev:transfinite}. 

\medskip

We can also replace the basis $\calB$ with the subset of monomials in $\CC[w]$ ordered by grevlex.    Let 
$$
Z_K(\alpha):=\inf\{\|p\|_K\colon p\in\CC[w],\ p(w) = w^{\alpha} +\lot \}.
$$
This is also submultiplicative and the analogue of Proposition \ref{prop:11} holds, replacing $Y_K$ by $Z_K$ and $T_{\calB}(K,\theta)$ by $T_{\CC[w]}(K,\theta)$, the Chebyshev transform of $K$ with respect to the monomials in $\CC[w]$, defined by the same kind of limit.

\begin{lemma}\label{lem:13}
\begin{enumerate}
\item We have $Z_K(\alpha)\geq Y_K(\alpha)$ for each $\alpha\in\NN_0^2$, and hence 
$$T_{\CC[w]}(K,\theta)\geq T_{\calB}(K,\theta) \hbox{ for each }\theta\in \Sigma^{\circ}.$$
\item Let $K_w$ denote the orthogonal projection of $K$ to the $w$-axis.  Then 
$$T_{\CC[w]}(K,\theta)=T_{\CC[w]}(K_w,\theta) \hbox{ for each } \theta\in \Sigma^{\circ}.$$
\end{enumerate}
\end{lemma}

\begin{proof}
Fix $\alpha\in\NN_0^2$ and let $p(w)=w^{\alpha} + \lot$  be a polynomial in $\CC[w]$ that satisfies $\|p\|_K=Z_K(\alpha)$.  Since $p\in\CC[w,z]_V$ also, $\|p\|_K\geq Y_K(\alpha)$ by definition.  The inequality of Chebyshev transforms follows by taking $|\alpha|$-th roots and a limit.  This proves the first statement.

For the second statement, observe that the $z$-components of a point do not enter into computation when evaluating polynomials in $\CC[w]$.  
\end{proof}


We return to $\calB$ again.  Given $\beta\in\calI$, let 
$$Y_{K,\beta}(\alpha):= \inf\{\|p\|_K\colon p\in\CC[w,z]_V,\ p(w,z) = w^{\alpha}z^{\beta} +  \lot  \}.$$ Using property ($\star$),  we show that $Y_{K,\beta}$ gives the same limit, independent of $\beta$. 

\begin{proposition}\label{prop:14}
Let $\theta\in \Sigma^{\circ}$ and $\beta\in\calI$.  Then 
$$
\lim_{\substack{|\alpha|\to\infty\\ \alpha/|\alpha|\to\theta}} Y_{K,\beta}(\alpha)^{1/|\alpha|} \ = \ T_{\calB}(K,\theta).
$$
Moreover, we have the convergence (put $\theta(t)=(t,1-t)$)
\begin{equation}\label{eqn:intconv}
\frac{1}{h(s)}\sum_{|\alpha|=s} \log Y_{K,\beta}(\alpha)^{1/|\alpha|} \ \to \ \int_0^1\log T_{\calB}(K,\theta(t))\, dt \quad\hbox{as }s\to\infty.
\end{equation}
\end{proposition}

\begin{proof} Fix $\beta\in\calI$.  
We first show that 
\begin{equation}\label{eqn:prop13<}  \limsup_{\substack{|\alpha|\to\infty\\ \alpha/|\alpha|\to\theta}} Y_{K,\beta}(\alpha)^{1/|\alpha|} \leq  T_{\calB}(K,\theta).\end{equation} 

 Denote the lim sup on the left-hand side of (\ref{eqn:prop13<}) by $L$.  Let $\alpha_{(1)},\alpha_{(2)},\ldots$ be a sequence of bi-indices such that as $n\to\infty$, 
$$|\alpha_{(n)}|\to \infty,\ \alpha_{(n)}/|\alpha_{(n)}|\to\theta,\ \hbox{and } 
Y_{K,\beta}(\alpha_{(n)})^{1/|\alpha_{(n)}|} \to L.$$ 
 For each $n$, let $p_n=w^{\alpha_{(n)}}+\lot$ be a polynomial satisfying $\|p_n\|_K=Y_K(\alpha_{(n)})$, and let $q_n:=z^{\beta}p_n$.  Then $q_n=w^{\alpha_{(n)}}z^{\beta}+\lot$, so 
\begin{equation}\label{eqn:12}
Y_{K,\beta}(\alpha_{(n)})\leq \|q_n\|_K \leq \|p_n\|_K\|z^{\beta}\|_K = Y_K(\alpha_{(n)})\|z^{\beta}\|_K. 
\end{equation}
If $\|z^{\beta}\|_K>0$ then $\|z^{\beta}\|_K^{1/|\alpha_{(n)}|}\to 1$ as $n\to\infty$.  Taking $|\alpha_{(n)}|$-th roots in the above inequality, then the limit as $n\to\infty$, we obtain  (\ref{eqn:prop13<}).  

Using (\ref{eqn:12}), we also calculate that 
$$
\frac{1}{sh(s)}\sum_{|\alpha|=s}\log Y_{K,\beta}(\alpha) \  \leq \    \frac{1}{sh(s)}\Bigl(\sum_{|\alpha|=s}\log Y_{K}(\alpha)\Bigr)  \ +  \ \frac{|\beta|}{s}\log(\|z\|_K). 
$$
Taking the lim sup on both sides as $s\to\infty$, and using Proposition \ref{prop:11}, we obtain
\begin{equation}\label{eqn:intlimsup}
\limsup_{s\to\infty} \frac{1}{sh(s)}\sum_{|\alpha|=s}\log Y_{K,\beta}(\alpha) \ \leq \ \int_0^1 \log T_{\calB}(K,\theta(t))\, dt. 
\end{equation}

If $\|z^{\beta}\|_K=0$, we may handle this by letting  $q_n:=(z^{\beta}+1)p_n$ and replacing $z^{\beta}$ in the above arguments with $z^{\beta}+1$.  

We now show that
\begin{equation}\label{eqn:prop13>}  \liminf_{\substack{|\alpha|\to\infty\\ \alpha/|\alpha|\to\theta}} Y_{K,\beta}(\alpha)^{1/|\alpha|} \geq  T_{\calB}(K,\theta)\end{equation}
using a similar argument.  Denote the lim inf by $L$.

  Let $\alpha_{(1)},\alpha_{(2)},\ldots$ be such that
$$
|\alpha_{(n)}|\to \infty,\ \alpha_{(n)}/|\alpha_{(n)}|\to\theta,\ \hbox{and } Y_{K,\beta}(\alpha_{(n)})^{1/|\alpha_{(n)}|}\to  L.  
$$
For each $n$, let $p_n=z^{\beta}w^{\alpha_{(n)}} +\lot$ be a polynomial satisfying $\|p_n\|_K=Y_{K,\beta}(\alpha_{(n)})$.  Let $q_n:=C_{\beta}^{-1}z^{\tilde\beta}p_n$, where $\tilde\beta$ is chosen as in $(\star)$ so that $z^{\beta}\cdot z^{\tilde\beta}=C_{\beta}w^{\gamma}+\lot$ for some $\gamma$.  Then by construction, $q_n=w^{\alpha_{(n)}+\gamma}+\lot$, so
\begin{equation}\label{eqn:14}
Y_K(\alpha_{(n)}+\gamma)\leq\|q_n\|_K\leq \|p\|_K C_{\beta}^{-1}\|z^{\tilde\beta}\|_K = Y_{K,\beta}(\alpha_{(n)})\cdot C_{\beta}^{-1}\|z^{\tilde\beta}\|_K.\end{equation}
Now take $|\alpha_{(n)}|$-th roots in the above inequality and let $n\to\infty$.  Since $\gamma$ is fixed, it is easy to see that $(\alpha_{(n)}+\gamma)/|\alpha_{(n)}+\gamma|\to\theta$ as $n\to\infty$.  On the left-hand side we have 
$$
 \lim_{n\to\infty}Y_K(\alpha_{(n)}+\gamma)^{1/|\alpha_{(n)}|}=  \lim_{n\to\infty} Y_K(\alpha_{(n)}+\gamma)^{1/|\alpha_{(n)}+\gamma|}=T_{\calB}(K,\theta).
$$

If $\|z^{\tilde\beta}\|_K>0$, then as before, $(Y_{K,\beta}(\alpha_{(n)})C_{\beta}^{-1}\|z^{\tilde\beta}\|_K)^{1/|\alpha_{(n)}|}\to L$ as $n\to\infty$.   This proves (\ref{eqn:prop13>}).    Then using (\ref{eqn:14}), a similar calculation as before yields
$$
\liminf_{s\to\infty} \frac{1}{sh(s)}\sum_{|\alpha|=s}\log Y_{K,\beta}(\alpha) \ \geq \ \int_0^1 \log T_{\calB}(K,\theta(t))\, dt. 
$$ 
Together with (\ref{eqn:intlimsup}),  we obtain the convergence (\ref{eqn:intconv}).  
(Again if $\|z^{\beta}\|_K=0$, use $z^{\tilde\beta}+1$ in place of $z^{\tilde\beta}$.) 
\end{proof}

\section{Group invariance} \label{sec:invariance}

Let $V$ be defined as in (\ref{eqn:f}), (\ref{eqn:V}).  The projection to $w$ exhibits $V$ as a finite branched cover over $\CC^2$. 
 Let $D\subseteq\CC^2$ be the complement of the branch locus, let $\Omega:=\{(w,z)\in V\colon w\in D\}$, and let $D_z:=\{z\in\CC^2\colon (w,z)\in\Omega\}$.    

In what follows, we will assume $V$ to be irreducible, hence $\Omega$, $D$ and $D_z$ are connected (\cite{chirka:complex}, \S 5.3).  
Given paths $\alpha,\beta:[0,1]\to D$ with  $\alpha(1)=\beta(0)$, denote by $\alpha\beta$ the concatenated path, 
$$\alpha\beta(t)=\left\{ \begin{array}{rl}\alpha(2t) &\hbox{if } t\in[0,1/2]\\ \beta(2t-1) &\hbox{if }t\in[1/2,1]
\end{array}\right. . $$

  Let $\pi(D)$ denote the fundamental group; if $\alpha,\beta$ are loops based at the same point, then $[\alpha][\beta]=[\alpha\beta]$.  (Here $[\cdot]$  means to take the homotopy class.)

We describe the monodromy action of the fundamental group $\pi(D)$.  Let $(w,z)\in V$ and assume for the moment that $w\in D$.  Consider a loop $\gamma\colon[0,1]\to D$  with  $w=\gamma(0)=\gamma(1)$.  Let $(w,z)\in V$ be a preimage of $w$ under the projection.  Lift $\gamma$ to a path $\hat\gamma$ in $\Omega$ with $\hat\gamma(0)=(w,z)$ and set
\begin{equation}\label{eqn:actionV}
[\gamma]\cdot (w,z):=\hat\gamma(1).
\end{equation}
 Since the lifting is locally holomorphic, this formula is well-defined by the monodromy theorem in several complex variables (i.e., if $\eta$ is another such loop at $w$ with lift $\hat\eta$, then $[\gamma]=[\eta]$ implies $\hat\gamma(1)=\hat\eta(1)$).   
  We get a group action  on $\Omega$. 

Write $(w,\tilde z):=[\gamma]\cdot(w,z)$.  Projecting to $z$ coordinates, put 
\begin{equation}\label{eqn:actionC}[\gamma]\cdot z:=\tilde z. \end{equation}
This map is locally holomorphic and bounded on $D_z$ for fixed $\gamma$.  The set $\CC^2\setminus D_z$ is locally a hypersurface.  By a theorem of Tsuji on removable singularities of holomorphic functions in several variables   
(see \cite{tsuji:removable}, Theorem 3),  $z\mapsto [\gamma]\cdot z$ extends holomorphically to all of $\CC^2$.  Hence the action of $\pi(D)$ given by (\ref{eqn:actionC}) extends to an action on  $\CC^2$.  Adjoining $w$-coordinates, the action of $\pi(D)$ on $\Omega$ given by (\ref{eqn:actionV}) extends to an action on $V$. 

\smallskip

Before continuing we recall some general notions of invariance.

\begin{definition}\label{def:p_G} \rm
Suppose $G$ is a finite group that acts on a complex manifold $X$,
$$G\times X\ni (g,z)\mapsto g\cdot z\in X.$$  
A function $f\colon X\to\CC$ is \emph{invariant (under $G$)} if $f(g\cdot z)=f(z)$ for all $g\in G$.  

Given $p\colon X\to\CC$ define $p_G\colon X \to \CC$ by 
$$
p_G(z):= \frac{1}{|G|}\sum_{g\in G} p(g\cdot z).
$$
\end{definition}

If $z\mapsto g\cdot z$ is holomorphic for each $g\in G$ then $p_G$ is also holomorphic. Denote by $\O(X)$ the holomorphic functions on $X$.  Then $\O(X)\ni p\mapsto p_G\in\O(X)$ 
is a linear projection from $\O(X)$ onto its subspace of invariant functions.

We will also call a set $E\subset X$  \emph{invariant (under $G$)} if $g\cdot E\subseteq E$ for all $g\in G$.  

\begin{lemma}\label{lem:22}
Let $E\subset X$ be invariant and $p:X\to\CC$.  Then $\|p_G\|_E\leq \|p\|_E$.
\end{lemma}

\begin{proof}
Assume $\|p\|_E<\infty$ otherwise the inequality holds trivially.  Let $x\in E$.  If $g\in G$ then $g\cdot x\in E$ by invariance, so $|p(g\cdot x)|\leq\|p\|_E$ and 
$$
|p_G(x)|\leq \frac{1}{|G|}\sum_{g\in G} |p(g\cdot x)| \leq \frac{1}{|G|}\sum_{g\in G} \|p\|_E  = \|p\|_E.
$$
Now take the sup over all $x\in E$.
\end{proof}

We now return to the $\pi(D)$-actions on $\CC^2$ and $V$.

\begin{lemma}\label{lem:43}
Let $f\colon V\to \CC$ be a function that is locally the restriction of a holomorphic function in a neighborhood (in $\CC^2\times\CC^2$) of every point of $V$.  Fix $a\in D$ and let $b_{(1)},\ldots,b_{(d)}$ be the corresponding points in $D_z$, i.e., $V\cap\{(w,z)\colon w=a\}=\{(a,b_{(1)}),\ldots,(a,b_{(d)})\}$.  Then for any $j,k\in\{1,\ldots,d\}$, 
$$
f_{\pi(D)}(a,b_{(j)})=f_{\pi(D)}(a,b_{(k)}).
$$
\end{lemma}

\begin{proof}
Fix $j,k\in \{1,\ldots,d\}$.  There is a path $\hat\eta$ from $(a,b_{(j)})$ to $(a,b_{(k)})$ in $\Omega$, which projects to a loop $\eta$ in $D$  based at $a$.   Hence $[\eta]\cdot b_{(j)}=b_{(k)}$ and 
$$\begin{aligned}
f_{\pi(D)}(a,b_{(k)}) = \frac{1}{|\pi(D)|}\sum_{[\gamma]\in\pi(D)} f(a,[\gamma]\cdot b_{(k)}) &= \frac{1}{|\pi(D)|}\sum_{[\gamma]\in\pi(D)} f(a,[\gamma]\cdot[\eta]\cdot b_{(j)})\\
&= \frac{1}{|\pi(D)|}\sum_{[\gamma\eta]\in\pi(D)} f(a,[\gamma \eta]\cdot b_{(j)})  \\  &= f_{\pi(D)}(a,b_{(j)}).  
\end{aligned}$$
\end{proof}


\begin{proposition}\label{prop:24}
Let $p\in\CC[w,z]_V$.   Then $p_{\pi(D)}\in\CC[w]$ and $\deg(p_{\pi(D)})\leq \deg(p)$.  
\end{proposition}

\begin{proof}
By Lemma \ref{lem:43}, the function $(w,z)\mapsto p_{\pi(D)}(w,z)$ on $\Omega$ is independent of $z$, so $p_{\pi(D)}(w,z)=\varphi(w)$ for some $\varphi\in\O(D)$.  Since $\varphi$ is locally bounded, it extends to an entire function on 
$\CC^2$ by Tsuji's removable singularity theorem \cite{tsuji:removable}.  Since $|z|=o(|w|)$ as $|w|\to\infty$ with $(w,z)\in V$, $\varphi$ grows at most polynomially in $w$; in fact, $|\varphi|=O(|w|^{|\deg p|})$ as $|w|\to\infty$.  This follows from applying $w^{\alpha}z^{\beta}=o(|w|^{|\alpha+\beta|})$ (valid for any monomial considered locally as a function in $w$) to the monomials of $p_{\pi(D)}$.    
By standard complex analysis arguments in $\CC^2$, the coefficients of the power series in $w$ for $\varphi$ must be zero for all powers greater than $\deg(p)$.  Hence it is a polynomial, i.e., an element of $\CC[w]\subseteq\CC[w,z]_V$. By uniqueness of normal forms, this polynomial (of degree at most $\deg(p)$) is precisely $p_{\pi(D)}$. \end{proof}


We next look at invariant sets in $V$ and $\CC^2$ under the action of $\pi(D)$.  Given a set $K\subset V$, denote by $K_w$ and $K_z$ the projections to $w$ and $z$ coordinates respectively.  (In particular, $\Omega_w=D$.)

\begin{proposition} \label{prop:25} Let $K\subset V$. The following are equivalent.
\begin{enumerate}[(1)]
\item $K$ is invariant under the $\pi(D)$-action given by (\ref{eqn:actionV}).
\item $K_z$ is invariant under the  $\pi(D)$-action given by (\ref{eqn:actionC}).
\item $K_z=f^{-1}(K_w)$.
\end{enumerate}
\end{proposition}

\begin{proof} We first prove the proposition when $K\subseteq\Omega$. (Hence $K_w\subseteq D$ and elements of $\pi(D)$ may be obtained from  loops based at points of $K_w$.) 
\begin{description}
\item[\rm (1) $\Leftrightarrow$ (2)] By adjoining and dropping $w$-coordinates, we get the equivalence of these two statements.
\item[\rm (2) $\Rightarrow$ (3)] We have $K_z\subseteq f^{-1}(K_w)$ by definition, so we need to show that invariance implies $f^{-1}(K_w)\subseteq K_z$.

Let $b\in f^{-1}(K_w)$.  Then $f(b)\in K_w$.  Hence there exists $a\in\CC^2$ such that $(f(b),a)\in K$.  Moreover, $f(a)=f(b)$ since $K\subset V$.  Let $\hat\eta$ be a path from $(f(b),a)$ to $(f(b),b)$ in $V$. This projects to a loop $\eta$ based at $f(b)$.  Now $a\in K_z$ by definition, and therefore $[\eta]\cdot a=b\in K_z$ by invariance. So $f^{-1}(K_w)\subseteq K_z$.

\item[\rm (3) $\Rightarrow$ (2)] Suppose $K_z=f^{-1}(K_w)$.  Let $b\in K_z$ and $[\eta]\in\pi(D)$, where we can take $\eta$ to be a loop based at $f(b)$.  Then $[\eta]\cdot b=a$ where $f(a)=b$.  So $a\in f^{-1}(K_w)$, i.e., $[\eta]\cdot b\in K_z$.  This says that $K_z$ is invariant.  
\end{description}

Now suppose each point of $K$ is a limit point of the set $K\cap\Omega$.     By continuity of the extended action,  statement 1 holds for $K\cap\Omega$ if and only if it holds for $K$; similarly, statement 2 holds for $K_z\cap\Omega_z$ if and only if it holds for $K_z$.  Also, $K_z\cap\Omega_z=f^{-1}(K_w\cap D)$ by continuity of $f$.  Hence the proposition holds for $K$ because it holds for $K\cap\Omega$ by the first part of the proof.

For general $K$, write $K=\bigcap_{j=1}^{\infty} K_j$ where the sets are decreasing in $j$:
$$K_j:=\{(w,z)\in V\colon |w-\tilde w|\leq 1/j \hbox{ whenever } (\tilde w,\tilde z)\in K  \}.$$  Each $K_j$ satisfies the conditions of the previous paragraph, so the proposition holds for these sets.  Moreover, using continuity of the $\pi(D)$-action and continuity of $f$, it can be shown by standard analysis that 
\begin{itemize}
\item $K_z$ is invariant if and only if $(K_j)_z$ is invariant for sufficiently large $j$; and
\item $K_z=f^{-1}(K_w)$ if and only if $(K_j)_z=f^{-1}((K_j)_w)$ for sufficiently large $j$. 
\end{itemize}
Hence the proposition holds for $K$ by taking a limit as $j\to\infty$.
\end{proof}





Our main theorem of this section is the following.

\begin{theorem} \label{thm:invariance}
Let $K\subset V$ be a compact $\pi(D)$-invariant set.  Then
$$
T_{\calB}(K,\theta) = T_{\CC[w]}(K,\theta) \hbox{ for each } \theta\in \Sigma^{\circ}.
$$
\end{theorem}

\begin{proof}
By Lemma \ref{lem:13} it suffices to show that 
 $T_{\CC[w]}(K,\theta)\leq T_{\calB}(K,\theta)$ for any $\theta\in \Sigma^{\circ}$.  This will follow immediately from showing that $Z_K(\alpha)\leq Y_K(\alpha)$ for any $\alpha\in\NN_0^2$.  

Fix $\alpha$ and let $p(w)=w^{\alpha}+\lot$ be a polynomial in $\CC[w,z]_V$ with the property that $\|p\|_K = Y_K(\alpha)$.  Let $q(w):=p(w)-w^{\alpha}$ denote the polynomial given by the lower order terms.  By Proposition \ref{prop:24},  $p_{\pi(D)},q_{\pi(D)}\in\CC[w]$ and 
$$
p_{\pi(D)}(w) = (w^{\alpha})_{\pi(D)} + q_{\pi(D)}(w) = w^{\alpha}+\lot.
$$
By the definition of $Z_K(\alpha)$ together with Lemma \ref{lem:22},
$$
Y_K(\alpha)=\|p\|_K\geq\|p_{\pi(D)}\|_K\geq Z_K(\alpha)
$$
which was to be shown.
\end{proof}

We will return to $\pi(D)$-invariant sets at the end of the next section (see equation (\ref{eqn:Lgraph})) and  throughout Section \ref{sec:pullback}.

\section{Transfinite diameter}\label{sec:transfinite}

We recall the construction of transfinite diameter.  Let $\calB$ denote a countable collection of monomials that span a  subspace $A\subset\CC[w,z]$.   (Later, $A$ will be one of $\CC[w,z]_V$, $\CC[w]$, or $\CC[z]$.)  
 We will assume that $\calB=\{\be_j\}_{j\in\NN}$ lists the elements of $\calB$ with respect to some ordering  $\prec$,  i.e., $\be_j\prec\be_k$ whenever $j\leq k$.  

Let $X\subset\CC^2\times\CC^2$ be an algebraic variety.   We will suppose that there is a filtration $A=\bigcup_{k=0}^{\infty}A_k$, where $A_0\subset A_1\subset\cdots$ is an increasing sequence of finite-dimensional linear subspaces with the property that $A_kA_l\subseteq A_{k+l}$.  This means that 
given 
$p_1\in A_k$ and  $p_2\in A_l$,  there exists $p_{12}\in A_{k+l}$ such that $$p_1(w,z)p_2(w,z)=p_{12}(w,z) \hbox{ for all } (w,z)\in X.$$   Denote the filtration (equivalently, the sequence $\{A_k\}_{k=0}^{\infty}$) by $\bf A$.  We require $\bf A$  to be compatible with the ordering in the sense that if $p\in A_k$ and $q\in A\setminus A_k$ then $p\prec q$.   For each $k$, let $\calB_k=\calB\cap A_k$.

Given a finite set of points $\{\zeta_j\}_{j=1}^{n}\subset X$, define
\begin{equation}\label{eqn:vandef}
\Van_{\calB}(\zeta_1,\ldots,\zeta_n) := \det\left[ \be_j(\zeta_k) \right]_{j,k=1}^n
\end{equation}
and given a compact set $K\subset X$ and $n\in\NN$, define
\begin{equation}\label{eqn:vanKdef}
\Van_{\calB,n}(K):=\max\{ |\Van_{\calB}(\zeta_1,\ldots,\zeta_{n})| \colon \zeta_j\in K\hbox{ for all } j\}.
\end{equation}

  Let $m_k$ denote the number of elements in $\calB_k$ (equivalently, the dimension of $A_k$), and let $l_k:=\sum_{\nu=1}^k \nu(m_{\nu}-m_{\nu-1})$.   
The \emph{transfinite diameter of $K$} is defined as
$$
d_{\bA,\calB}(K) := \limsup_{n\to\infty} \left(\Van_{\calB,m_n}(K)\right)^{1/l_n}.
$$

\begin{example}\rm Let $K\subset\CC^2$, $\bA =\{\CC[z]_{\leq k}\}_{k=0}^{\infty}$ and $\calB=\{z^{\alpha}\colon \alpha\in\NN_0^2 \}$.   Then $d_{\bA,\calB}(K)$ is the Fekete-Leja transfinite diameter in $\CC^2$, which we will denote simply by $d(K)$.
\end{example}

\begin{example} \label{ex:KinV}  \rm 
We will consider the following cases of transfinite diameter on $$V=\{(w,z)\in\CC^2\times\CC^2\colon w=f(z)\},$$
where $f:\CC^2\to\CC^2$ is a polynomial map of degree $d\geq 1$ in each component as in Example (\ref{ex:13}).   Let $K\subset V$ be a compact set.
\begin{enumerate}

\item Let $\bA=\{\CC[z]_{\leq k}\}$  and $\calB=\{z^{\alpha}\}$.  We obtain the Fekete-Leja transfinite diameter of the projection to $z$, $d_{\bA,\calB}(K) = d(K_z).$  Elements of $\CC[z]$ may be identified with elements of $\CC[V]$ by elimination: the map $f^*$ defined by 
\begin{equation}\label{eqn:iso}\CC[w,z]\ni p(w,z)\longmapsto p(f(z),z)=:(f^*p)(z)\in\CC[z]\end{equation} descends to an isomorphism $f^*\colon\CC[V]\to\CC[z]$. 

\item Consider $\CC[w,z]_V$; by uniqueness of normal forms, an element of $\CC[V]$ is represented by a unique $q\in\CC[w,z]_V$.  Hence the restriction of the map (\ref{eqn:iso}) to  $\CC[w,z]_V$ is an isomorphism, which we use to define the the elements of $\bA=\{A_k\}$:  
    $$ A_k:=\{q\in\CC[w,z]_V\colon  f^*q \in\CC[z]_{\leq dk}\}.$$ 
Let $\calB$ be the monomial basis of $\CC[w,z]_V$; we order it as in Example \ref{ex:21}.  (It is straightforward to show that the ordering is compatible with $\bA$.)   
The transfinite diameter $d_{\bA,\calB}(K)$ obtained here is different to the one above.  

\item  Let $K\subset V$, let $\bA=\{\CC[w]_{\leq k}\}$, and let $\calB=\{w^{\alpha}\}$.  Note that $\CC[w]_{\leq k}=A_k\cap\CC[w]$ with $A_k$ as above.  Similar to case 1,
$
d_{\bA,\calB}(K) = d(K_w).
$  
\end{enumerate}
The relation among all these transfinite diameters will be important later.
\end{example}

\begin{example}\rm 
We give a concrete computation of $\bA$ in case 2 of the above example. Consider $w=f(z)$ given by
$$
w_1=z_1^2+z_2,\quad w_2=z_2^2+1.
$$
The monomials in $\CC[w,z]_V$ are $$\left\{w^{\alpha}z_1^mz_2^n\colon \alpha\in\NN_0^2;\  m,n\in\{0,1\} \right\}.$$
Then $A_1$ is spanned by the monomials $\calB_1=\{1,z_1,z_2,z_1z_2,w_1,w_2\}$ which correspond to 
$\{1,z_1,z_2,z_1z_2, z_1^2+z_2,z_2^2+1\}$ in $\CC[z]_{\leq 2}$.  In general, $A_k$ is spanned by the monomials 
$$
\calB_k=\{ w^{\alpha}, w^{\beta}z_1, w^{\gamma}z_2,w^{\delta}z_1z_2 \hbox{ where } |\alpha|\leq k \hbox{ and }
|\beta|,|\gamma|,|\delta|\leq k-1 \}.
$$                                                                                                                                                                                                                                                                                                                                                                                                                                                                                                                                                                                                                                                                                                                                                                                                                                                                                                                                                                                                                                                                                                                                                                                                                                                                                                                                                                                                                                                                                                                                                                                                                                                                                                                                                                                                                                                                                                                                                                                                                                                                                                                                                                                                                                                                                                                                                                                                                                                                                                                                                                                                                                                                                                                                                                                                                                                                                                                                                                                                                                                                                                                                                                                                                                                                                                                                                                                                                                                                                                                                                                                                                                                                                                                                                                                                                                                                                                                                                                                                                                                                                                                                                                                                                                                              This can be verified by a straightforward induction.
\end{example}
    
Denote by $d^{(1)}(K)$, $d^{(2)}(K)$ and $d^{(3)}(K)$ the transfinite diameters defined in Example \ref{ex:KinV}.



\begin{theorem} \label{thm:zaharj}
Let $K\subset V$ be compact, and let $\theta(t)=(t,1-t)$.  We have 
\begin{eqnarray}
d^{(1)}(K)   &=& \int_0^1 \log T_{\CC[z]}(K,\theta(t))\,dt  \ =  \ d(K_z),    \label{eqn:d1formula}    \\ 
d^{(2)}(K) &=& \int_0^1 \log T_{\calB}(K,\theta(t))\,dt,   \label{eqn:d2formula} \\
d^{(3)}(K) &=& \int_0^1 \log T_{\CC[w]}(K,\theta(t))\,dt  \ =  \ d(K_w),     \label{eqn:d3formula} 
\end{eqnarray}
where $\calB$ in (\ref{eqn:d2formula}) denotes the monomial basis of $\CC[w,z]_V$.  
\end{theorem}

\begin{proof}[Sketch of Proof.]
The derivation of these formulas uses the standard method of bounding ratios of Vandermonde determinants above and below.  Formulas (\ref{eqn:d1formula}) and (\ref{eqn:d3formula}) are the classical formula of Zaharjuta for the sets $K_z$ and $K_w$ in $\CC^2$.  
  We sketch the proof of (\ref{eqn:d2formula}), giving the main ideas but omitting a number of technical details.  

First, we have the inequality
\begin{equation}\label{eqn:vratio}
Y_{K,\beta(n)}(\alpha(n))\ \leq\    \frac{\Van_{\calB,n}(K)}{\Van_{\calB,n-1}(K)}\ \leq \  n Y_{K,\beta(n-1)}(\alpha(n-1))
\end{equation}
provided that $\Van_{\calB,n-1}(K)>0$.  The lower bound in (\ref{eqn:vratio}) can be established as follows. Let $\zeta_1,\ldots,\zeta_{n-1}\in K$, and write 
\begin{equation}\label{eqn:van11}
\Van_{\calB}(\zeta_1,\ldots,\zeta_{n-1},(w,z)) = \Van_{\calB}(\zeta_1,\ldots,\zeta_{n-1})\be_{n}(w,z) +\lot,
\end{equation}
where we expand (\ref{eqn:vandef})  down the last column.  If the points $\zeta_j$ are chosen to attain the sup in (\ref{eqn:vanKdef}) for $\Van_{\calB,n-1}(K)$, and this quantity is nonzero, we may factor it out of both terms on the right-hand side of (\ref{eqn:van11}) to obtain 
\begin{equation*}
\Van_{\calB,n-1}(K)\left(\be_{n}(w,z) +\lot \right) = \Van_{\calB}(\zeta_1,\ldots,\zeta_{n-1},(w,z)).
\end{equation*}
Denote the polynomial in parethesis on the left-hand side by $$p(w,z)=\be_{n}(w,z) +\lot=w^{\alpha(n)}z^{\beta(n)}+\lot$$  and put $s_{\beta,n}:=|\alpha(n)|.$ 
Evaluate $p$ at a point $(w,z)=\zeta_n$ for which $|p(\zeta_n)|=\|p\|_K$ to obtain
$$
\Van_{\calB,n-1}(K)Y_{K,\beta(n)}(\alpha(n))\leq  \Van_{\calB,n-1}(K)\|p\|_K = \Van_{\calB}(\zeta_1,\ldots,\zeta_{n-1},\zeta_n)\leq \Van_{\calB,n}(K),
$$
and we have the lower bound in (\ref{eqn:vratio}).   

The upper bound is obtained by expanding the determinant in a different way; we omit the details.

Now (assuming all quantities are nonzero, see also the next remark) we use (\ref{eqn:vratio}) to estimate the telescoping product $$\dfrac{\Van_{\calB,m_n}(K)}{\Van_{\calB,m_n-1}(K)}\dfrac{\Van_{\calB,m_n-1}(K)}{\Van_{\calB,m_n-2}(K)}\cdots\dfrac{\Van_{\calB,m_{n-1}+1}(K)}{\Van_{\calB,m_{n-1}}(K)}$$ and obtain
\begin{equation}\label{eqn:van<<}
\prod_{k=m_{n-1}+1}^{m_n}Y_{K,\beta(k)}(\alpha(k))  \ \leq \  \frac{\Van_{\calB,m_n}(K)}{\Van_{\calB,m_{n-1}}(K)} \ \leq \  
 \frac{m_n!}{m_{n-1}!} \prod_{k=m_{n-1}}^{m_n-1}Y_{K,\beta(k)}(\alpha(k)) .
\end{equation}
Regrouping and relabelling terms, we rewrite the products  on each side of (\ref{eqn:van<<}) as 
\begin{equation} \label{eqn:prods}\begin{aligned}
\prod_{k=m_{n-1}+1}^{m_n}Y_{K,\beta(k)}(\alpha(k)) &= \prod_{\beta\in\calI} \Bigl(\prod_{|\alpha|=s_{\beta,n}} Y_{K,\beta}(\alpha) \Bigr),   \\ 
\prod_{k=m_{n-1}}^{m_n-1}Y_{K,\beta(k)}(\alpha(k))  &= C_n \prod_{\beta\in\calI} \Bigl(\prod_{|\alpha|=s_{\beta,n}} Y_{K,\beta}(\alpha) \Bigr) \quad\hbox{where } C_n=\tfrac{Y_{K,\beta(m_n)}(\alpha(m_n))}{Y_{K,\beta(m_{n-1})}(\alpha(m_{n-1}))}.   
\end{aligned}\end{equation}
 Then $s_{\beta,n}\to \infty$ as $n\to\infty$,  and by Proposition \ref{prop:14},
\begin{equation}\label{eqn:limch}
\frac{1}{ s_{\beta,n}h(s_{\beta,n})} \sum_{|\alpha|=s_{\beta,n}} \log Y_{K,\beta}(\alpha) \to \ \int_0^1\log T_{\calB}(K,\theta(t))\, dt \quad\hbox{as }n\to\infty.
\end{equation}
Fix $\epsilon>0$; then (\ref{eqn:limch}) implies
\begin{equation}\label{eqn:36}
\prod_{\beta\in\calI}\Bigl(\prod_{|\alpha|=s_{\beta,n}} Y_{K,\beta}(\alpha) \Bigr) 
\leq  
(1+\epsilon)^{s_n}\exp\left( \int_0^1\log T_{\calB}(K,\theta(t))\, dt \right)^{s_n|\calI|}, 
\end{equation}
for sufficiently large $n$, where $s_n=\sum_{\beta\in\calI}s_{\beta,n}$. 

\medskip

 Consider the upper bound in (\ref{eqn:van<<}).   Using (\ref{eqn:36})  and  (\ref{eqn:prods}), 
\begin{equation}\label{eqn:van<}
\frac{\Van_{\calB,m_n}(K)}{\Van_{\calB,m_{n-1}}(K)} \ \leq \ \frac{m_n!}{m_{n-1}!}C_n(1+\epsilon)^{s_n}\exp\left( \int_0^1\log T_{\calB}(K,\theta(t))\, dt \right)^{s_n|\calI|}
\end{equation}
for sufficiently large $n$, say $n\geq N$.  Now apply (\ref{eqn:van<}) to the telescoping product
$$
\frac{\Van_{\calB,m_n}(K)}{\Van_{\calB,m_{n-1}}(K)}\frac{\Van_{\calB,m_{n-1}}(K)}{\Van_{\calB,m_{n-2}}(K)}\cdots
\frac{\Van_{\calB,m_{N+1}}(K)}{\Van_{\calB,m_{N}}(K)}\Van_{\calB,m_{N}}(K)
$$
to obtain the estimate 
\begin{equation}\label{eqn:vanp<}
\Van_{\calB,m_n}(K) \leq \Van_{\calB,m_{N}}(K)\frac{m_n!}{m_N!} D_{N,n}(1+\epsilon)^{t_{N,n}}   \exp\left( \int_0^1\log T_{\calB}(K,\theta(t))\, dt \right)^{t_{N,n}|\calI|}
\end{equation}
where 
$$D_{N,n}=\prod_{k=N}^nC_k=\dfrac{Y_{K,\beta(m_n)}(\alpha(m_n))}{Y_{K,\beta(m_{N})}(\alpha(m_{N}))} 
\  \hbox{ and }  \   t_{N,n}=\sum_{k=N}^ns_k.
$$
We take $l_n$-th roots on both sides of (\ref{eqn:vanp<}) and apply the following limits (calculations omitted) as $n\to\infty$:
$$
(m_n!)^{1/l_n}\to 1,\  \ D_{N,n}^{1/l_n}\to 1,\   \     \frac{t_{N,n}|\calI|}{l_n}\to 1.
$$
This yields
$$
\limsup_{n\to\infty}  (\Van_{\calB,m_n}(K))^{1/l_n} \leq (1+\epsilon)^{\frac{1}{|\calI|}}\exp\left(
\int_0^1\log T_{\calB}(K,\theta(t))\, dt\right),$$
and since $\epsilon>0$ was arbitrary,
$$
\limsup_{n\to\infty}  (\Van_{\calB,m_n}(K))^{1/l_n} \leq \exp\left(
\int_0^1\log T_{\calB}(K,\theta(t))\, dt\right).
$$

By a similar argument as above, this time using the lower bound in (\ref{eqn:van<<}),  
$$
\liminf_{n\to\infty}  (\Van_{\calB,m_n}(K))^{1/l_n} \geq \exp\left(
\int_0^1\log T_{\calB}(K,\theta(t))\, dt\right).
$$
These last two inequalities yield (\ref{eqn:d2formula}).
\end{proof}

\begin{remark}\rm
If $\Van_{\calB,n-1}(K)>0$ and $\Van_{\calB,n}=0$  then $Y_{K,\beta(n)}(\alpha(n))=0$ by (\ref{eqn:vratio}).  One can use a simple inductive argument to verify that $\Van_{\calB,m}(K)=0$ and $Y_{K,\beta(m)}(\alpha(m))=0$ for all $m\geq n$, so equation (\ref{eqn:d2formula}) holds trivially in this case.
\end{remark}

Fix a compact set $K\subset\CC^2$ and let 
\begin{equation}\label{eqn:Lgraph} L:=\{(w,z)\in V\colon w\in K\}. \end{equation} Then
$
L_w=K$  and  $L_z=f^{-1}(K)$.

\begin{corollary}\label{cor:35}
Let $K\subset\CC^2$ be compact, and let $L\subset V$ be defined as in (\ref{eqn:Lgraph}).  Then
$d^{(2)}(L)=d^{(3)}(L)=d(K)$.  
\end{corollary}

\begin{proof} 
We have $K=L_w$ and $L_z=f^{-1}(L_w)$.  By Proposition \ref{prop:25}, $L$ is invariant under the group action of $\pi(D)$ described  in Section \ref{sec:invariance}.  Hence by Theorem \ref{thm:invariance}, $T_{\calB}(L,\theta(t))=T_{\CC[w]}(L,\theta(t))$ for all $t\in(0,1)$, where $\theta(t)=(t,1-t)$.  Finally, by formulas (\ref{eqn:d2formula}) and (\ref{eqn:d3formula}) of the previous theorem, $d^{(2)}(L)=d^{(3)}(L)$.  
\end{proof}

The relation with $d^{(1)}(L)$ will be studied in the next section.

\section{Pullback formula}\label{sec:pullback}


  Let $d\in\NN$ and $(w_1,w_2)=f(z_1,z_2)$ where $f=(f_1,f_2)$ is a regular polynomial mapping of degree $d$.  This means that $\hat f^{-1}(0,0)=\{(0,0)\}$, where $\hat f = (\hat f_1,\hat f_2)$ denotes the leading homogeneous part of $f$.    Its components are of the form
\begin{equation} \label{eqn:fhat} \begin{aligned}
\hat f_1(z) &= a_{d}z_1^d +a_{d-1}z_1^{d-1}z_2+\cdots+a_{0}z_2^d,  \\
\hat f_2(z)  &= b_{d}z_1^d +b_{d-1}z_1^{d-1}z_2+\cdots+b_{0}z_2^d.
\end{aligned}\end{equation}

\begin{remark}\label{rem:reg} \rm
Note that $\hat f_1,\hat f_2$ are products of linear factors, whose zero sets are lines through the origin. When $f$ is not regular, $\hat f^{-1}(0,0)$ must contain at least one complex line through the origin which is the zero set of a common factor.
\end{remark}

We will prove the following special case (in dimension 2) of the main result of \cite{demarcorumely:transfinite}.

\begin{theorem}\label{thm:resthm}
Let $d\in\NN$ and let $f:\CC^2\to\CC^2$ be a regular polynomial map of degree $d$. For any compact set $K\subset\CC^2$, we have the formula
$$
d(f^{-1}(K)) = |\Res(\hat f)|^{-1/(2d^2)}d(K)^{1/d}.
$$ 
\end{theorem}

On both sides, $d$ denotes the Fekete-Leja transfinite diameter.   
The term $\Res(\hat f)$ on the right-hand side is the resultant of $\hat f$, which  is the $2d\times 2d$  determinant 
$$
\Res(\hat f) \ = \  \Res(\hat f_1,\hat f_2) \ := \    
\det \begin{bmatrix}
a_d & a_{d-1} & a_{d-2} & \cdots   & a_1&a_0 & \  & \  & \ \\
\  & a_d & a_{d-1} & a_{d-2}  & \cdots  &a_1& a_0     & \  \\
\  & \  & \ & \ddots & \   & \  & \   & \ddots &\   \\
\ & \ & \ & \ & a_d & a_{d-1} & a_{d-2} & \cdots & a_0    \\
b_d & b_{d-1} & b_{d-2} & \cdots   & b_1&b_0 & \  & \  & \ \\
\  & b_d & b_{d-1} & b_{d-2}  & \cdots  &b_1& b_0     & \  \\
\  & \  & \  & \ddots & \   & \  & \  & \ddots & \   \\
\ & \ & \ & \ & b_d & b_{d-1} & b_{d-2} & \cdots & b_0  
\end{bmatrix} , 
$$
where $a_{j},b_j$ are defined as in (\ref{eqn:fhat}) and the triangular regions above the $a_0,b_0$ diagonals and below the $a_d,b_d$ diagonals are filled with zeros.

 We relate the transfinite diameters  $d(f^{-1}(K)),d(K)$ to transfinite diameters on the graph $$V=\{(w,z)\colon w=f(z)\}\subset\CC^2\times\CC^2.$$
Also, write $\hat V = \{(w,z)\colon w=\hat f(z)\}$.  

Let $K\subset\CC^2$ be compact, and let $L\subset V$ be defined as in (\ref{eqn:Lgraph}).  Then 
$$
d(f^{-1}(K))=d(L_z)=d^{(1)}(L),\quad d(K) = d(L_w)=d^{(3)}(L)=d^{(2)}(L)
$$
where we use Corollary \ref{cor:35}.  Hence we need to relate $d^{(1)}(L)$ to $d^{(2)}(L)$. 

 We consider a filtration of the monomials in $\CC[w,z]_V$.   Similar to (\ref{eqn:iso}),  let  
$\hat f^*\colon\CC[w,z]\to\CC[z]$ be defined by $\hat f^*(p(w,z)) = p(\hat f(z),z)$.     Let
$$
\calF_{n}=\{   w^{\alpha}z^{\beta}\in\CC[w,z]_V \colon \hat f^*(w^{\alpha}z^{\beta})\in\CC[z]_{\leq n}\}.
$$
We also let $\calF_{m,n}:=\calF_{n}\setminus\calF_m$ when $m<n$; similarly, define $\calB_{m,n}$ where $\calB$ is the monomial basis of $\CC[w,z]_{V}$.   

\begin{notation}\rm
Given a collection $\calM$ of monomials in $\CC[w,z]$, put $\CC[w,z]_{\calM}:=\span(\calM)$.  

Hence  $A_k=\CC[w,z]_{\calB_k}$, $\CC[w,z]_{\calF_{m,n}}=\CC[w,z]_{\calF_n}\setminus \CC[w,z]_{\calF_m}$, etc.
\end{notation}

The following lemma is a straightforward consequence of the definitions.  
\begin{lemma}
\begin{enumerate}\item 
There are $n+1$ monomials in $\calF_{n-1,n}$, given by 
$$
\{w^{\alpha}z^{\beta}\in\CC[w,z]_V\colon d|\alpha|+|\beta|=n\}.
$$
\item For each $k\in\NN$,   $\calB_k = \calF_{dk}$.  
 \end{enumerate} \qed
\end{lemma}

\begin{remark}\rm
We will assume in what follows that $f$ is as in Example \ref{ex:13}, where the powers of $z$ in $\calB$ are given by monomials not in $\langle z_1z_2^{d-1},z_2^d\rangle$, i.e., the powers of $z_2$ are as small as possible.  Pre- and post-composing with generic rotations  $R_1,R_2$ will transform $f$ into a mapping $g:=R_2\circ f\circ R_1$ as in that example.  This changes the set $f^{-1}(K)$ in Theorem \ref{thm:resthm} to $ g^{-1}(K)$, and the resultant to $\Res(\hat g)=\Res(R_2\circ\hat f\circ R_1)$.  We have $d(f^{-1}(K)) = d(g^{-1}(K))$ by invariance of the Fekete-Leja transfinite diameter under a rotation, and  $$\Res(\hat g)=\Res(R_2\circ\hat f\circ R_1)=\det(R_2)^d\Res(\hat f)\det(R_1)^{d^2}=\Res(\hat f)$$ by standard properties of resultants and the fact that $\det(R_1)=\det(R_2)=1$.  Hence the formula for $f$ follows from the formula for $g$, so we can reduce to this case. \end{remark}

To simplify calculations it will be convenient to use a basis of monomials that is slightly different to $\calB$. 
 We define another collection of monomials.  When $n\in\{0,\ldots,2d-2\}$ put $\calG_n=\calF_n$.  Otherwise, 
\begin{equation}\label{eqn:dkr} n=(d-1)+kd+r \hbox{ for some }k\in\NN \hbox{ and }r\in\{0,\ldots,d-1\}.\end{equation}  Then we inductively define
$$
\calG_n := \calG_{n-1}\cup   \{w^{\alpha}z_1^{j} z_2^{d-1-j}\colon |\alpha|=k, j=0,\ldots,d-1\} $$
if $r=0$,  and
$$
\calG_n:= \calG_{n-1} \cup \{ w^{\alpha}z_1^{d-1-j+r} z_2^{j}\colon |\alpha|=k,\  j=0,\ldots,d-1\} \cup \{z_1^{j}z_2^{n-j}\colon j=0,\ldots,r-1      \} $$ 
if $r\neq 0$.  

Put $\calG_{m,n}=\calG_{n}\setminus\calG_{m}$.  Following the pattern in the above lemma, define 
$\calC_k:=\calG_{dk}$ for each $k\in\NN$ and $\calC:=\bigcup_k \calC_k$.  
Observe that (parts of) the bases can be constructed inductively.   For example, when $n$ is sufficiently large,
$$ 
\calF_{n+d-1,n+d}=w_1\calF_{n-1,n}\cup w_2\calF_{n-1,n}, \quad \tilde\calG_{n+d-1,n+d}=w_1\tilde\calG_{n-1,n}\cup w_2\tilde\calG_{n-,n}
$$
where $\tilde\calG_{n-1,n}=\calG_{n-1,n}\setminus\{z_1z_2^{n-j}\colon j=0,\ldots,r-1\}$.    (Note that in both cases there is lots of overlap in the two pieces of the union.) A specific illustration is given a bit later in Example \ref{ex:d=2ex}. 

The elements of $\calG_{n-1,n}$, constructed at each stage, form the correct number for a basis. The issue is whether they are linearly independent. We have the following lemma which relies on the fact that $f$ is regular.
\begin{lemma}\label{lem:reg}
The collection of monomials $\calG_{2d-1}$ is linearly independent.
\end{lemma}
\begin{proof}
When $n<2d-1$, $\calG_n=\calF_n$, i.e., the monomials are normal forms in $\CC[w,z]_V$ which are automatically linearly independent.   Hence to study the linear independence of $\calG_{2d-1}$, we only need to study the linear independence of 
$$\calG_{2d-2,2d-1} = \{w_1z_1^{d-1},w_1z_1^{d-2}z_2,\ldots,w_1z_2^{d-1},w_2z_1^{d-1},w_2z_1^{d-2}z_2,\ldots,w_2z_2^{d-1}\}.$$
We need to verify that any linear combination of monomials does not reduce to a linear combination of monomials in  $\calF_{2d-2}$.   Substituting for $w_1,w_2$ in the above elements, we obtain the collection of polynomials
$$
\{ f_1(z)z_1^{d-1},\ldots, f_1(z)z_2^{d-1}, f_2(z)z_1^{d-1},\ldots, f_2(z)z_2^{d-1}\}\subset\CC[z]_{2d-1}.
$$
We need to show that no linear combination of these polynomials has degree strictly less than $2d-1$.  This is equivalent to showing that if 
\begin{equation}\label{eqn:l66}
c_1z_1^{d-1}\hat f_1(z)+\cdots + c_dz_2^{d-1}\hat f_1(z) + c_{d+1}z_1^{d-1}\hat f_2(z)+\cdots + c_{2d}z_2^{d-1}\hat f_1(z)=0 
\end{equation}
for all $z$, then $c_j=0$ for all $j=1,\ldots,2d$.  If not, then the left-hand side of (\ref{eqn:l66}) is a nonzero homogeneous polynomial, given by a  product of linear factors at least one of which is nonzero. Clearly we cannot factor out either $z_1$ or $z_2$ from this sum, so the nonzero linear factor must be a common factor of $\hat f_1$ and $\hat f_2$.  This contradicts the fact that $f$ is regular.  
\end{proof}

For larger values of $n$, we can use the lemma to verify that the monomials in $\calG_n$ are linearly independent by an inductive argument based on the step by step construction described above.  We omit the details. 


\begin{example} \label{ex:d=2ex}\rm
We illustrate these bases using a simple example. Let $d=2$ and suppose $\calI=\{1,z_1,z_2,z_1^2\}$.  (This is the generic case, recall Example \ref{ex:13}.)  Then
$$ 
\calF_{4}=\{1,z_1,z_2,w_1,w_2,z_1^2,w_1z_1,w_1z_2,w_2z_1,w_2z_2,w_1z_1^2,w_2z_1^2,w_1^2,w_1w_2,w_2^2\}$$ while
$$
\calG_4 =\{1,z_1,z_2,w_1,w_2,z_1^2,w_1z_1,w_1z_2,w_2z_1,w_2z_2,w_1z_1^2,w_1z_1z_2,w_2z_1^2,w_2z_1z_2,z_2^4\}     .$$  
The last four monomials (which are in $\calF_{3,4}$ and $\calG_{3,4}$ respectively) are slightly different.  
At the next level, 
$$\begin{aligned}
\calF_{4,5}=w_1\calF_{2,3}\cup w_2\calF_{2,3}, &\qquad \calG_{4,5}=w_1\calG_{2,3}\cup w_2\calG_{2,3}, \\
\calF_{5,6}=w_1\calF_{3,4}\cup w_2\calF_{3,4}, &\qquad \calG_{5,6}=w_1\tilde\calG_{3,4}\cup w_2\tilde\calG_{3,4}\cup\{z_2^6\}
\end{aligned}$$
where $\tilde\calG_{3,4}=\calG_{3,4}\setminus\{z_2^4\}$.  
More generally, one can show inductively that 
$$\begin{aligned}
\calF_{n,n+1}\  = \bigcup_{|\alpha|=n-1}\!\! w^{\alpha}\calF_{2,3},&\quad 
\calG_{n,n+1} \  = \bigcup_{|\alpha|=n-1}\!\! w^{\alpha}\calG_{2,3} \quad\hbox{if $n$ is even}, \\
\calF_{n,n+1}\  = \bigcup_{|\alpha|=n-1}\!\! w^{\alpha}\calF_{3,4},&\quad 
\calG_{n,n+1} \  = \bigcup_{|\alpha|=n-1}\!\! w^{\alpha}\tilde\calG_{3,4}\cup \{z_2^{n+1}\} \quad\hbox{if $n$ is odd}.
\end{aligned}$$
Observe that
$$
\calG_{2,3}=\calF_{2,3}=\{w_1z_1,w_1z_2,w_2z_1,w_2z_2\}, 
$$
so that when $n$ is even, $\calG_{n,n+1}=\calF_{n,n+1}$. When $n$ is odd, apart from $z_2^{n+1}$, one can use the inductive formula to match up monomials $w^{\alpha}z^{\beta}\in\calG_{n,n+1}$ with monomials $w^{\alpha'}z^{\beta'}\in\calF_{n,n+1}$ so that the differences $|\alpha-\alpha'|$ and $|\beta-\beta'|$ are uniformly bounded above by the maximum difference in powers between pairs of monomials in $\calF_{3,4}\cup\calG_{3,4}$.  
\end{example}

\begin{example}\label{ex:gend}\rm
For a slightly more general example, let $d\geq 2$ and consider $n=2d$ (i.e., $k=1$ and $r=1$ in (\ref{eqn:dkr})).  We have
$$\begin{aligned}
\calG_{2d-1,2d} &=\{w_1z_1^{d},\ldots,w_1z_1z_2^{d-1},w_2z_1^{d},\ldots,w_2z_1z_2^{d-1}, z_2^{2d}\}, \\
\calF_{2d-1,2d} &= \{w_1^2,w_1w_2,w_2^2, w_1z_1^{d},w_2z_1^{d},\ldots,w_1z_1^{2}z_2^{d-2},  w_2z_1^2z_2^{d-2}  \}.   
\end{aligned}$$
In both cases there are $2d+1$ monomials.  The restriction of $\hat f^*$ to $\CC[w,z]_{\calG_{2d-1,2d}}$ or $\CC[w,z]_{\calF_{2d-1,2d}}$ is an isomorphism onto $\CC[z]_{=2d}$.  
The elements of $\calF_{2d-1,2d}$ may be obtained from the elements of the basis $\{z_1^jz_2^{2d-j}\}_{j=0}^{2d}$ of $\CC[z]_{=2d}$ by computing normal forms in $\CC[w,z]_{\hat V}$ and reading off the monomials.  

Similar to the previous example, we have for $k>1$ that 
$$
\calF_{(k+1)d-1,(k+1)d} \ = \bigcup_{|\alpha|=k-1}\!\! w^{\alpha}\tilde\calF_{2d-1,2d},\quad 
\calG_{(k+1)d-1,(k+1)d} \  = \bigcup_{|\alpha|=k-1}\!\! w^{\alpha}\tilde\calG_{2d-1,2d}
\cup\{z_2^{(k+1)d}\}.
$$
where $\tilde\calG_{2d-1,2d}=\calG_{2d-1,2d}\setminus\{z_2^{2d}\}$.   Thus we can create a one-to-one mapping between all but one pair of monomials in $\calF_{(k+1)d-1,(k+1)d}$ and $\calG_{(k+1)d-1,(k+1)d}$ so that the difference in powers is bounded above by the maximum difference in powers between pairs of monomials in $\calF_{2d-1,2d}$ and $\calG_{2d-1,2d}$.  

Similar calculations as above hold for $r\in\{0,1,\ldots,d-1\}$.  

\end{example}

We have the following calculations (proof omitted). 

\begin{lemma}\label{lem:limC}
Let $m_n$ be the number of elements in $\calC_n$ and let $l_n=\sum_{\nu=1}^n \nu(m_{\nu}-m_{\nu-1})$.  If $d\geq 2$ then 
$$\begin{aligned}
m_n &= \sum_{\nu=0}^{nd} (\nu+1) = \frac{n^2d^2}{2}+o(n^2) , \\
l_n  &=\sum_{\nu=1}^n \nu( \nu d^2-\frac{d(d-3)}{2}) = \frac{n^3d^2}{3} + o(n^3)
\end{aligned}$$
as $n\to\infty$.   In particular, $m_n=o(l_n)$. 
\qed 
\end{lemma}

Define $\Van_{\calC,n}(\zeta_1,\ldots,\zeta_n)$, $\Van_{\calC,n}(K)$ similar to (\ref{eqn:vandef}), (\ref{eqn:vanKdef}), using the basis $\calC$.  We need a lemma whose detailed proof is quite technical and tedious.  We sketch a proof and write out the details for a particular example.  The basic idea behind the proof is that the monomials in $\calB$ and $\calC$ are the same except for a few factors that become  negligible as $n\to\infty$.

\medskip

In the next lemma, recall the set $L$ defined in (\ref{eqn:Lgraph}).  
\begin{lemma}\label{lem:66}
We have $$\lim_{n\to\infty} (\Van_{\calC,m_n}(L))^{1/l_n} = \lim_{n\to\infty} (\Van_{\calB,m_n}(L))^{1/l_n}.$$
 (Recall that the right-hand side equals $d^{(2)}(L)$.)  
\end{lemma}

\begin{proof}[Sketch of Proof]  
As before, suppose 
$$n=(d-1) +kd+r, \hbox{ where } k\in\NN   \hbox{ and } r\in\{0,\ldots,d-1\}.$$ 
For each $j=2d-2,\ldots,n$, all but at most $d-1$ elements of 
$\calG_{j-1,j}$ have total degree in $z$  uniformly bounded above by $2d$.   (For $\calG_{n-1,n}$ these are the last $r$ elements.)  All elements of $\calF_{j-1,j}$ have total degree in $z$ uniformly bounded above by $d-1$.  

Consider replacing rows of the Vandermonde matrix for $\Van_{\calC}(\zeta_1,\ldots,\zeta_{m_n})$ with rows of the matrix for $\Van_{\calB}(\zeta_1,\ldots,\zeta_{m_n})$.  If a row of the form $[w^{\alpha}z^{\beta}(\zeta_1) \cdots w^{\alpha}z^{\beta}(\zeta_{m_n})]$ is replaced by one of the form $[w^{\tilde\alpha}z^{\tilde\beta}(\zeta_1)\cdots w^{\tilde\alpha}z^{\tilde\beta}(\zeta_{m_n})]$, this changes the determinant by a factor of $\prod_{j=1}^{m_n}w^{\tilde\alpha-\alpha}z^{\tilde\beta-\beta}(\zeta_j)$.  

We may compare, one by one, rows of $\Van_{\calC}$ corresponding to  monomials  $w^{\alpha}z^{\beta}\in\calG_{j-1,j}$ 
with rows of $\Van_{\calB}$ corresponding to monomials $w^{\tilde\alpha}z^{\tilde\beta}\in\calF_{j-1,j}$ for which the differences $\tilde\alpha-\alpha$ and $\tilde\beta-\beta$ are uniformly bounded, as in Examples \ref{ex:d=2ex} and \ref{ex:gend} .  We do this for all $j=2d-2,\ldots,n$.  If $L$ avoids the coordinate axes in $z$ and $w$ one can find constants $c,C$ such that 
\begin{equation}\label{eqn:diff}
c \leq  |w^{\tilde\alpha-\alpha}z^{\tilde\beta-\beta}(\zeta)| \leq C \end{equation}
whenever $\zeta\in L$.\footnote{The condition on the coordinate axes may be removed by doing further analysis at the end---apply the result to $L\setminus\{(w,z)\colon |z|,|w|<\epsilon\}$ and then let $\epsilon\to 0$.}

We bound the ratio $|\Van_{\calC}(\zeta_1,\ldots,\zeta_{m_n})|/|\Van_{\calB}(\zeta_1,\ldots,\zeta_{m_n})|$ using powers of $c$ and $C$ by applying (\ref{eqn:diff}) to each row that we compare.    This estimate works for all rows except those which contain no powers of $w$; altogether, fewer than  $m_n=o(l_n)$ rows. 

 The remaining rows of $\calG_{n}$ contain no powers of $w$; for each $j=2d-2,\ldots,n$ this is at most $d-1$ rows of $\calG_{j-1,j}$; so in total, $O((d-1)n)$ rows.  For each of these rows we can use an estimate of the form $c^{O(n)},C^{O(n)}$ (perhaps replacing $c,C$ as appropriate).  The sum of all the exponents is $O((d-1)n^2)= o(l_n)$ by Lemma \ref{lem:limC}.  

Putting the estimates in the previous two paragraphs together, we obtain 
$$c^{o(l_n)/l_n}\leq \left(\frac{|\Van_{\calC}(\zeta_1,\ldots,\zeta_{m_n})|}{|\Van_{\calB}(\zeta_1,\ldots,\zeta_{m_n})|}\right)^{1/l_n}\leq C^{o(l_n)/l_n}$$
and the upper and lower bounds go to $1$  as $n\to\infty$.  The result follows.
\end{proof}

\begin{example}\label{ex:d=2}\rm 
To illustrate  a specific case, suppose $d=2$ as in Example \ref{ex:d=2ex}.  
When $n\in\NN$, 
$$\begin{aligned}
\calF_{2n-1,2n} &= \bigcup_{|\alpha|=n-1} w^{\alpha}\calF_{3,4} =    \{ w^{\alpha}\colon |\alpha|=n\}\cup\{w^{\alpha}z_1^2\colon |\alpha|=n-1\}, \\
\calG_{2n-1,2n} &= \bigcup_{|\alpha|=n-1}w^{\alpha}\tilde\calG_{3,4} \cup\{z_2^{2n}\} 
 \\ & =        \{ w^{\alpha}z_1^2\colon |\alpha|=n-1\}\cup \{w^{\alpha}z_1z_2\colon |\alpha|=n-1 \}  \cup\{z_2^{2n}\}.  
\end{aligned}$$
We compare monomials in 
$$
\calF_{2n-1,2n}\setminus\calG_{2n-1,2n} = \{w^{\alpha}\colon |\alpha|=n\} = \{w^{\alpha}w_1\colon |\alpha|=n-1\}\cup\{w_2^n\}
$$
to those in 
$$\calG_{2n-1,2n}\setminus\calF_{2n-1,2n} = \{w^{\alpha}z_1z_2\colon|\alpha|=n-1\}\cup\{z_2^{2n}\}.$$
We have 
$$
c \leq  \left|\frac{w_1(\zeta)}{z_1z_2(\zeta)}\right| \leq C \quad \hbox{for all }\zeta\in K 
$$
where $c>0$ is smaller than the minimum and $C$ is larger than the maximum of $|w_1|/|z_1z_2|$ on $K$.  We use this estimate to compare the  row of $\Van_{\calB}$ corresponding to $w^{\alpha}w_1$ with that of $\Van_{\calC}$ corresponding to $w^{\alpha}z_1z_2$, for each $\alpha$ with $|\alpha|=n-1$.   The only other row to compare is that of $w_2^{n}\in\calF_{2n-1,2n}$ with that of $z_2^{2n}\in\calG_{2n-1,2n}$.  We can use the estimate $c^{n}\leq |w_2^n/z_2^{2n}|\leq C^n$ here if we choose 
$$0<c\leq\min\{|w_2|/|z_2^2|,  |w_1|/|z_1z_2| \},\  C\geq\max\{|w_2|/|z_2^2|,  |w_1|/|z_1z_2| \}.$$  
Altogether, we obtain $2n$ powers of $c,C$ from rows corresponding to monomials in $\calF_{2n-1,2n}$ and $\calG_{2n-1,2n}$. 

As shown in Example \ref{ex:d=2ex},  $\calF_{2n,2n+1}=\calG_{2n,2n+1}$ and no estimates are needed.  Putting everything  together, let $k$ be a large integer; then $k\in\{2n,2n+1\}$ for some $n\in\NN$ and   
$$
c^{\nu_k}  \leq    \left(\frac{|\Van_{\calC}(\zeta_1,\ldots,\zeta_{m_k})|}{|\Van_{\calB}(\zeta_1,\ldots,\zeta_{m_k})|}\right)  \leq C^{\nu_k}
$$
where $$\nu_k=\sum_{j=1}^n 2j < \sum_{j=1}^{k}2j< 2m_k=o(l_k).$$  So the $l_k$-th roots of these quantities go to 1 as $k\to\infty$. 
\end{example}

We will need the following elementary lemma which follows immediately from properties of determinants under row operations.

\begin{lemma}\label{lem:hatf}
Let $t\in\NN$. Consider polynomials $q_j(z)= f_j(z) + r_j(z)$ for each $j=1\ldots,t$.  Let $s>t$ and consider square matrices given by evaluating polynomials at a set of points $\{\zeta_1,\ldots,\zeta_s\}$,  
$$
F = \left[\begin{array}{ccc}
R(\zeta_1) & \cdots & R(\zeta_s) \\ \hline
f_1(\zeta_1)  & \cdots & f_1(\zeta_s)  \\
f_2(\zeta_2)  & \cdots & f_2(\zeta_s)  \\
\vdots & \ddots & \vdots \\
f_t(\zeta_1) & \cdots &  f_t(\zeta_s) 
\end{array}\right],\  
G = \left[\begin{array}{ccc}
R(\zeta_1) & \cdots & R(\zeta_s) \\ \hline
 q_1(\zeta_1) & \cdots &  q_1(\zeta_s) \\
 q_2(\zeta_1) & \cdots &  q_2(\zeta_s) \\
\vdots & \ddots & \vdots \\
 q_t(\zeta_1) & \cdots &   q_t(\zeta_s)
\end{array}\right],
$$
where $R(z)$ is a vector of $s-t$ monomials.  If every monomial of $r_j(z)$ is an entry of $R(z)$ (for all $j$) then $\det F=\det G$.\qed
\end{lemma}

In what follows let $\calA:=\{z^{\alpha}\}$ denote the monomial basis of $\CC[z]$ under a graded ordering, so that in the  proposition below, the rows of the Vandermonde matrix for $\Van_{\calA}(\zeta_1,\ldots,\zeta_{m_n})$  are given by monomials in $\CC[z]_{\leq dn}$, for each $n\in\NN$.  (For convenience, local computations will use grevlex with $z_1\prec z_2$, but reordering monomials will not affect the result.)

\begin{proposition}\label{prop:613}
Fix an integer $n>d-1$ and a set of points  $\{\zeta_1,\ldots,\zeta_{m_n}\}$.  Then 
$$|\Van_{\calC}(\zeta_1,\ldots,\zeta_{m_n})| 
=  O(1)^{o(n^3)}   |\Res(\hat f)|^{\frac{dn^3}{6}+o(n^3)} | \Van_{\calA}(\zeta_1,\ldots,\zeta_{m_n})|$$
\end{proposition}

\begin{proof}
We study the transformation of Vandermonde matrix columns.  Consider the $j$-th column of the matrix for $\Van_{\calC}(\zeta_1,\ldots,\zeta_{m_n})$; its entries  are of the form $w^{\alpha}z^{\beta}(\zeta_j)$.  We transform these entries to linear combinations of expressions of the form $z^{\gamma}(\zeta_j)$, where $|\gamma|\leq dn$, in a step by step process.  In the calculations that follow we assume that 
\begin{equation}\label{eqn:assumeh} w_1=f_1(z)=\hat f_1(z),\quad w_2=f_2(z)=\hat f_2(z).\end{equation}

Let $d-1< k\leq dn$ and consider the entries corresponding to monomials in $\calG_{k,k-1}$.  Write \begin{equation}\label{eqn:kdlr}  k =(d-1)+\ell d+r\end{equation} 
where $\ell\in\{0,1,\ldots,n-1   \}$ and $r\in\{0,\ldots,d-1\}$.  
Write
$$
\bW_k= \left[\begin{array}{c}
  \bw_0(\zeta_j)\\ \cdots \\ \bw_{\ell}(\zeta_j)
\end{array}\right],
\quad\hbox{where }
 \bw_{s}(\zeta_j) = \left[\begin{array}{c} 
  w_1^{\ell-s}w_2^sz_1^{r+d-1}(\zeta_j)\\ w_1^{\ell-s}w_2^sz_1^{r+d-2}z_2(\zeta_j) \\ \cdots \\
w_1^{\ell-s}w_2^{s}z_1^{r}z_2^{d-1}(\zeta_j)   
\end{array}\right]   \hbox{ for } s\in\{0,\ldots,\ell\}.
$$

For convenience in what follows, we suppress the dependence on $\zeta_j$.  By using (\ref{eqn:assumeh}) and (\ref{eqn:fhat}) we may transform a pair of factors $w_1,w_2$ in the monomials contained in $\bw_0,\bw_1$, respectively, as follows:
$$\begin{aligned}
\left[\begin{array}{c}\bw_0\\ \hline \bw_1
\end{array}\right]
=\left[\begin{array}{c}
w_1^{\ell}z_1^{r+d-1}\\ 
\cdots \\
w_1^{\ell}z_1^rz_2^{d-1}\\ \hline 
w_1^{\ell-1}w_2z_1^{r+d-1} \\ 
\cdots \\ w_1^{\ell-1}w_2z_1^{r}z_2^{d-1}
\end{array}\right]
&=\left[\begin{array}{c}
w_1\cdot w_1^{\ell-1}z_1^{r+d-1}\\ \cdots \\
w_1\cdot w_1^{\ell-1}z_1^rz_2^{d-1}\\ \hline
w_2\cdot w_1^{\ell-1}z_1^{r+d-1}\\ \cdots \\
w_2\cdot w_1^{\ell-1}z_1^rz_2^{d-1}
\end{array}\right] \\
&= \left[\begin{array}{c}
(a_dz_1^d+ \cdots +a_0z_2^d) w_1^{\ell-1}z_1^{r+d-1}  \\ \cdots\cdots \\
(a_dz_1^d+ \cdots +a_0z_2^d) w_1^{\ell-1}z_1^rz_2^{d-1} \\ \hline
(b_dz_1^d+ \cdots +b_0z_2^d) w_1^{\ell-1}z_1^{r+d-1}   \\ \cdots\cdots \\
(b_dz_1^d+ \cdots +b_0z_2^d) w_1^{\ell-1}z_1^rz_2^{d-1} 
\end{array}\right] \\
&=\left[\begin{array}{ccccc}
a_d  & \cdots   & a_0 & \  &  \ \\
\  &  \ddots & \      & \ddots &\   \\
\ & \ &  a_d  & \cdots & a_0    \\   \hline 
b_d  & \cdots  &b_0 & \  &  \  \\
 \  & \ddots &  \  & \ddots & \   \\
 \ & \ & b_d  & \cdots & b_0  
\end{array}\right] 
\left[\begin{array}{c}
w_1^{\ell-1}z_1^{r+2d-1} \\ \cdots \\
w_1^{\ell-1}z_1^{r+d}z_2^{d-1} \\ \hline
w_1^{\ell-1}z_1^{r+d-1}z_2^d \\ \cdots \\
w_1^{\ell-1}z_1^rz_2^{r+2d-1} 
\end{array}\right]  \\
&=:  R\left[\begin{array}{c}
\bw_{0}[w_1] \\ \hline \bw_{1}[w_2] 
\end{array}\right]  .
\end{aligned}$$
  Adjoining $\bw_2,\ldots,\bw_{\ell}$ to this calculation, we obtain
$$
\left[\begin{array}{c} \bw_0\\ \bw_1\\ \hline 
\bw_2\\  \cdots\\  \bw_{\ell}  \end{array}\right]  
= \left[\begin{array}{c|c}
R & \b0 \\ \hline
\b0 & I
\end{array}\right] \left[\begin{array}{c} \bw_0[w_1]\\ \bw_1[w_2]\\ \hline 
\bw_2\\ \cdots\\  \bw_{\ell}  \end{array}\right]   =: R_1 \left[\begin{array}{c} \bw_0[w_1]\\ \bw_1[w_2]\\ \hline 
\bw_2\\ \cdots\\  \bw_{\ell}  \end{array}\right] .
$$
Observe that $R$ is the matrix of $\Res(\hat f)$ and $\det(R_1)=\det(R)= \Res(\hat f)$.

We may transform, by the same method, another pair of factors $w_1,w_2$ in a pair of respective blocks to monomials in $z_1,z_2$ of degree $d$; for example, transform $w_1$ in the block $\bw_2$ and $w_2$ in the block $\bw_3$.  A similar calculation as above yields 
$$
\left[\begin{array}{c} \bw_0[w_1]\\ \bw_1[w_2] \\  
\bw_2 \\ \bw_3\\ \cdots\\  \bw_{\ell}  \end{array}\right]  
=  R_2 \left[\begin{array}{c} \bw_0[w_1]\\ \bw_1[w_2]\\  
\bw_2[w_1]\\   \bw_3[w_2] \\ \cdots\\  \bw_{\ell}  \end{array}\right] 
$$
where $\det(R_2)=\Res(\hat f)$ also. 

Suppose in what follows that $\ell$ is odd.  Then there are an even number of blocks.    By pairing up blocks we can eliminate one factor each of $w_1$ and $w_2$ in $(\ell+1)/2$ steps; overall, the total degree in $w$ of each monomial is decreased by 1.   We continue transforming powers of $w$ to powers of $z$ in (disjoint) block pairs.  The transformation matrices $R_2,R_3,\ldots,$ etc.  satisfy $|\det(R_j)|=|\Res(\hat f)|$.  All powers of $w$ are eliminated after doing these $(\ell+1)/2$ steps $\ell$ times (to reduce the total degree in $w$ from $\ell$ to $0$).   The final expression on the right-hand side is 
$$
R_{\ell(\ell+1)/2}\begin{bmatrix}
\bw_0[w_1^{\ell}] \\ \bw_1[w_1^{\ell-1}w_2]\\ \cdots \\  \bw_{\ell}[w_2^{\ell}]  
\end{bmatrix} 
= R_{\ell(\ell+1)/2} \begin{bmatrix} z_1^{k}\\ z_1^{k-1}z_2\\ \cdots\\ z_1^{r+d-1}z_2^{k-r-d+1}
\end{bmatrix} 
$$
with $\det R_{\ell(\ell+1)/2}=\Res(\hat f)$.  Define the product $\bR_k:=R_{\ell(\ell+1)/2}\cdots R_2R_1$; then
$$
\begin{bmatrix}  w_1^{\ell}z_1^{r+d-1}\\ \cdots \\ w_2^{\ell}z_2^{r+d-1}\end{bmatrix}
= \bR_k \begin{bmatrix}
z_1^{k} \\ \cdots \\ z_1^{r+d-1}z_2^{k-r-d+1}
\end{bmatrix} 
$$
with \begin{equation*} |\det\bR_k|=\left|\Res(\hat f)\right|^{\ell(\ell+1)/2} =  \left|\Res(\hat f)\right|^{\frac{\ell^2}{2}+o(\ell^2)}.\end{equation*}


When $\ell$ is even, we first modify the basis monomials so that monomials containing $w$ have degree $\ell-1$.\footnote{A similar argument to the proof of Lemma \ref{lem:reg} (i.e. substituting $w=f(z)$) may be used to show that the new collection is linearly independent.}  We account for this modification with estimates as in (\ref{eqn:diff}).  We then do a similar calculation as in the odd case.  Altogether, we obtain a change of variable matrix $\bR_k$ with the property that
\begin{equation} \label{eqn:Respower}
|\det\bR_k| = O(1)^{o(\ell)}\left|\Res(\hat f)\right|^{\frac{\ell^2}{2}+o(\ell^2)}.
\end{equation}
More details are given in an example which follows the proof.

The above procedure is carried out on monomials in $\calG_{k,k-1}$ for all $k\in\{d,\ldots,dn\}$ to eliminate powers of $w$.  We obtain for the full Vandermonde matrix
\begin{equation} \label{eqn:WZ}
\bW  =  \left[\begin{array}{c}  
* \\ \hline  \bW_d \\ \cdots \\   \bW_{nd}
\end{array}\right] 
=\left[\begin{array}{c|ccc}
I & \ & \  &  \  \\  \hline 
\ & \bR_d & \  & \  \\
\  &\  & \ddots & \  \\
\  &\  &\  & \bR_{nd} 
\end{array}\right] \bZ  =: \bR\bZ ,
\end{equation}
where the entries of $\bZ$ are $z^{\alpha}(\zeta_j)$ for $|\alpha|\leq dn$.  For convenience, the monomials that give the components of $\bW$ have been reordered so that all monomials containing only powers of $z$ have been transferred to the top region (denoted by $*$).\footnote{This  reordering will have no effect since we are only interested in the absolute value of the determinant.}  

We now compute $|\det\bR| = \prod_{j=d}^{nd} |\det\bR_j|$.  For each $\ell=1,\ldots,n-1$ we have 
$$|\det\bR_{(\ell+1)d-1}|=|\det\bR_{(\ell+1)d}|=\cdots=|\det\bR_{(\ell+1)d+(d-2)}|=O(1)^{o(\ell)}|\Res(\hat f)|^{\frac{\ell^2}{2}+o(\ell^2)}. $$
The computation of each of the above $\bR_j$s involves the same calculation, using the same value of $\ell$, with $r=0,\ldots,d-1$.  This accounts for all but a few factors $|\det \bR_j|$ whose number stays bounded as $n\to\infty$.   Hence 
\begin{equation}\label{eqn:Sn1} |\det\bR|=O(1)^{o(S_n)}|\Res(\hat f)|^{S_n}\end{equation} where
\begin{equation}\label{eqn:Sn}
S_n = d\left(\sum_{\ell=1}^{n-1}\frac{\ell^2}{2} + o(\ell^2) \right) + O(1)  = d\left(\frac{n^3}{6} + o(n^3)\right). 
\end{equation}
Finally, apply (\ref{eqn:WZ}) to  all columns of the Vandermonde matrix. We obtain
$$
[\bW(\zeta_1) \cdots \bW(\zeta_{m_n})] = \bR[\bZ(\zeta_1)\cdots\bZ(\zeta_{m_n})] 
$$  
(using the points $\zeta_j$ to distinguish the columns).  
Now take the absolute value of the determinant on both sides and apply (\ref{eqn:Sn1}) and   (\ref{eqn:Sn}).  The result follows as long as (\ref{eqn:assumeh}) holds. 

In general, $f$ may consist of $\hat f$ plus lower order terms.  By calculations similar to the above, one can show that 
$$
[\bW(\zeta_j) + \lot\cdots \bW(\zeta_{m_n})+\lot ] =\bR[\bZ(\zeta_1)\cdots\bZ(\zeta_{m_n})].
$$
Given a row of the matrix on left-hand side whose leading term corresponds to some monomial, the monomials in $\lot$ correspond to rows further up the matrix.  We can  reduce the calculations in this case to those of the previous one by applying Lemma \ref{lem:hatf} and an inductive argument.
\end{proof}

\begin{example}\rm
We illustrate how powers of the resultant appear as a result of the change of variable calculation when $d=2$.  For simplicity, we consider when $f$ is homogeneous, i.e., 
$$\begin{aligned}
w_1&= f_1(z)=\hat f_1(z)=a_2z_1^2+a_1z_1z_2+a_0z_2^2,\\
w_2&=f_2(z)=\hat f_2(z)=b_2z_1^2+b_1z_1z_2+b_0z_2^2.
\end{aligned}$$
We compute the details of a case in which $\ell$ is even, which needs a first step to modify it into a similar form as the odd case.     Suppose $k=9$, so that $\ell=4$ and $r=0$ in (\ref{eqn:kdlr}).  The 10 rows of the Vandermonde determinant for $k=9$ correspond to 
$$\calG_{8,9}=\{w_1^4z_1, w_1^4z_2,  w_1^3w_2z_1, w_1^3w_2z_2, w_1^2w_2^2z_1, w_1^2w_2^2z_2, w_1w_2^3z_1, w_1w_2^3z_2, w_2^4z_1, w_2^4z_2      \}.$$
For the first step, we replace the above list of monomials by
$$
\mathcal{H} := \{ w_1^3z_1^3, w_1^3z_1^2z_2, w_1^2w_2z_1^3, w_1^2w_2z_1^2z_2,  w_1w_2^2z_1^3, w_1w_2^2z_1^2z_2, w_2^3z_1^3, w_2^3z_1^2z_2, z_1z_2^8,z_2^9\}.
$$
The first 8 monomials in $\mathcal{H}$ are off by a factor of  either $w_1/z_1^2$ or $w_2/z_1^2$ compared to those in $\calG_{8,9}$, and the last two monomials are off by a factor of $(w_2/z_2^2)^4$ each.  Using the maximum $C$ and minimum $c$ of the values on $K$ of the functions in the set $\{|w_1|/|z_1|^2,|w_2|/|z_1|^2,|w_2|/|z_2|^2\}$, the Vandermonde determinant with $\mathcal{H}$ replacing $\calG_{8,9}$ is off by a factor   bounded by $c^{16}$  from below and by  $C^{16}$ from above (cf. Example \ref{ex:d=2}).

We now carry out the change of variables from $w$ to $z$ for the first 8 monomials of 
$\mathcal{H}$, similar to the proof.  The first two steps yield
$$
\begin{bmatrix}
w_1^3z_1^3\\ w_1^3z_1^2z_2\\ w_1^2w_2z_1^3\\ w_1^2w_2z_1^2z_2\\   w_1w_2^2z_1^3\\  w_1w_2^2z_1^2z_2\\  w_2^3z_1^3\\  w_2^3z_1^2z_2 
\end{bmatrix}
=\begin{bmatrix}
a_2 & a_1 & a_0 &  & & & &  \\
& a_2 & a_1 & a_0 &  & & & \\
b_2 & b_1 & b_0 &  & & & &  \\
& b_2 & b_1 & b_0 &  & & & \\
 & & & &a_2 &a_1 &a_0 &  \\
 & & & & &a_2 &a_1 &a_0  \\
& & & &b_2 &b_1 &b_0 &  \\
 & & & & &b_2 &b_1 &b_0  
\end{bmatrix} \begin{bmatrix}
w_1^2z_1^5\\ w_1^2z_1^4z_2\\ w_1^2z_1^3z_2^2\\ w_1^2z_1^2z_2^3\\  
w_2^2z_1^5\\ w_2^2z_1^4z_2\\ w_2^2z_1^3z_2^2\\ w_2^2z_1^2z_2^3
\end{bmatrix}.
$$
(As before, empty spaces are filled with zeros.) The next two steps yield
$$
\begin{bmatrix}
w_1^2z_1^5\\ w_1^2z_1^4z_2\\ w_1^2z_1^3z_2^2\\ w_1^2z_1^2z_2^3\\  
w_2^2z_1^5\\ w_2^2z_1^4z_2\\ w_2^2z_1^3z_2^2\\ w_2^2z_1^2z_2^3
\end{bmatrix}
=\begin{bmatrix}
a_2 & a_1 & a_0 &  & & & &  \\
& a_2 & a_1 & a_0 &  & & & \\
 & & & &a_2 &a_1 &a_0 &  \\
 & & & & &a_2 &a_1 &a_0  \\
b_2 & b_1 & b_0 &  & & & &  \\
& b_2 & b_1 & b_0 &  & & & \\
& & & &b_2 &b_1 &b_0 &  \\
 & & & & &b_2 &b_1 &b_0  
\end{bmatrix} \begin{bmatrix}
w_1z_1^7\\ w_1 z_1^6z_2\\ w_1 z_1^5z_2^2\\ w_1 z_1^4z_2^3\\  
w_2 z_1^5z_2^2\\ w_2z_1^4z_2^3\\ w_2z_1^3z_2^4\\ w_2z_1^2z_2^5
\end{bmatrix},
$$
and the last two steps yield
$$
\begin{bmatrix}
w_1z_1^7\\ w_1 z_1^6z_2\\ w_1 z_1^5z_2^2\\ w_1 z_1^4z_2^3\\  
w_2 z_1^5z_2^2\\ w_2z_1^4z_2^3\\ w_2z_1^3z_2^4\\ w_2z_1^2z_2^5
\end{bmatrix}
=\begin{bmatrix}
a_2 & a_1 & a_0 &  & & & &  \\
& a_2 & a_1 & a_0 &  & & & \\
 & & & &a_2 &a_1 &a_0 &  \\
 & & & & &a_2 &a_1 &a_0  \\
b_2 & b_1 & b_0 &  & & & &  \\
& b_2 & b_1 & b_0 &  & & & \\
& & & &b_2 &b_1 &b_0 &  \\
 & & & & &b_2 &b_1 &b_0  
\end{bmatrix} \begin{bmatrix}
z_1^9\\ z_1^8z_2\\ z_1^7z_2^2\\ z_1^6z_2^3\\  
z_1^5z_2^4\\ z_1^4z_2^5\\ z_1^3z_2^6\\ z_1^2z_2^7
\end{bmatrix}.
$$
Altogether there are 6 copies of the resultant matrix.  Taking determinants, the right-hand side of equation (\ref{eqn:Respower}) for our example reads 
$O(1)^{16}|\Res(\hat f)|^6$. Note that the power of the $O(1)$ term is negligible compared to the power of the resultant for large $\ell$, although it happens to be larger here (where $\ell=4$).  The remaining two monomials, $z_1z_2^8$ and $z_2^9$ (whose rows are not shown), remain the same throughout.
\end{example}

Theorem \ref{thm:resthm} now follows by taking a limit.

\begin{proof}[Proof of Theorem \ref{thm:resthm}]
Suppose for each $n\in\NN$ we take Fekete points $\zeta_{1},\ldots,\zeta_{m_n}$ in $L$ (i.e., points that maximize the determinant). Then by the above Proposition, 
\begin{equation}\label{eqn:last}
|\Van_{\calC,m_n}(L)|^{1/l_n} = \left(O(1)^{o(n^3)}|\Res(\hat f)|^{\frac{dn^3}{6}+o(n)} |\Van_{\calA,m_n}(L)|\right)^{1/l_n}.
\end{equation}
By Lemma \ref{lem:limC} we have $l_n=\frac{n^3d^2}{3}+ o(n^3)$.  As $n\to\infty$,  
$$\begin{aligned}
O(1)^{o(n^3)/l_n} &\to 1, \\ 
|\Van_{\calC,m_n}(L)|^{1/l_n}& \to d^{(2)}(L) = d(K), \\
|\Van_{\calA,m_n}(L)|^{1/l_n} = \left(|\Van_{\calA,m_n}(K)|^{\frac{1}{ (nd)^3+o(n)}}\right)^{\frac{(nd)^3+o(n)}{l_n}}  &\to
\left(d^{(1)}(L)\right)^d = \left( d(f^{-1}(K))\right)^d, \\
|\Res(\hat f)|^{\left(\frac{dn^3}{6}+o(n)\right)/l_n}&\to |\Res(\hat f)|^{1/(2d) }.
\end{aligned} $$
(Note that in the third line, a root of $(nd)^3+o(n)$ is the appropriate one for $d^{(1)}(L)$, where we take the sum of the degrees of all monomials in  $\CC[z]_{\leq nd}$.)  The theorem follows from these limits by letting $n\to\infty$ on both sides of (\ref{eqn:last}).  
\end{proof}

Since the pullback formula depends only on the resultant of the leading homogeneous part of the polynomial mapping, we have the following immediate consequence,  which is a special case of Theorem 5 in \cite{bloomcalvi:multivariate}.
\begin{corollary}
Let $f\colon\CC^2\to\CC^2$ be a regular polynomial mapping with leading homogeneous part $\hat f$.  Then 
$d(f^{-1}(K))=d(\hat f^{-1}(K))$.   \qed
\end{corollary}

\section{Final Remarks}\label{sec:final}

It seems straightforward to adapt the method of this paper to relate transfinite diameters on $K$ and $f^{-1}(K)$ in  other situations, e.g.,
\begin{itemize}
\item when using more general notions of transfinite diameter, such as the so-called $C$-transfinite diameter defined using a convex body \cite{mau:transfinite}, or weighted transfinite diameter \cite{bloomlev:transfinite}; 
\item when $f=(f_1,f_2)$ is not necessarily regular, or $f_1,f_2$ are of different degrees.
\end{itemize}
The idea is to lift $K$ to the graph $V$ and study projections to $w$ and $z$.  The $z$ variables may be eliminated using symmetry and the $w$ variables may be eliminated using substitution; patterns in the coefficients due to substitution then give rise to normalization factors.  It would be interesting to see what other types of resultants arise in this way.


\bigskip

\noindent{\bf Acknowledgement.} I would like to thank the referee for pointing out an issue in the original proof of Proposition \ref{prop:613}. 



\end{document}